\documentclass[12pt]{article}
\usepackage{amsfonts}
\usepackage{graphicx}
\usepackage{latexsym,amsmath,color}

\def\esp{\mathbb{E}}

\def\1{\mathbb{I}}

\newcounter{thm}[section]
\newcounter{appen}

\newtheorem{theor}[thm]{Theorem}
\newtheorem{cor}[thm]{Corollary}
\newtheorem{lem}[thm]{Lemma}

\newenvironment{proof}[1][Proof]{\noindent \textbf{#1.}
}{\rule{0.5em}{0.5em}}
\newtheorem{hp}{Assumption}

\begin{document}


\title{Testing the predictor effect on a functional response}

\author{Valentin Patilea\footnote{CREST (Ensai) \& IRMAR, France; patilea@ensai.fr. This author
gratefully acknowledges financial support from the Romanian National Authority for Scientific Research, CNCS-UEFISCDI, project
PN-II-ID-PCE-2011-3-0893.}\;\;\;\;\;\; C\'esar S\'anchez-Sellero\footnote{Facultad de Matem\'aticas, Univ. de Santiago de Compostela, Spain; cesar.sanchez@usc.es. This author
gratefully acknowledges support from the Spanish Ministry of Science, project MTM2008-03010, and from Ensai.}\;\;\;\;\;\; Matthieu Saumard\footnote{Pontificia Univ. Catol\'{i}ca de Valpara\'{i}so,
Chile; Matthieu.Saumard@gmail.com. This author gratefully acknowledges financial support fom CONICYT/FONDECYT project number 3140602.}}
\date{\today}
\maketitle

\begin{abstract}

{\small This paper examines the problem of nonparametric testing for the no-effect of a random  covariate (or predictor) on a functional response. This means testing whether the conditional expectation of the response given the covariate is almost surely zero or not, without imposing any model relating response and covariate. The covariate could be univariate, multivariate or functional.
Our test statistic is a quadratic form involving univariate nearest neighbor smoothing and the asymptotic critical values are given by the standard normal law. When the covariate is multidimensional or functional, a preliminary dimension reduction device is used which allows the effect of the covariate to be summarized into a univariate random quantity.
The test is able to detect not only linear but nonparametric alternatives. The responses could have conditional variance of unknown form and the law of the covariate does not need to be known. An empirical study with simulated and real data shows that the test performs well in applications.

\bigskip

}

\end{abstract}



\section{Introduction}

The analysis of functional data has captured  huge attention from the scientific community in recent years. One main reason is that functional data arise in many applications. Regression models for functional data are now common tools for  practitioners. As an illustration, in this paper we reconsider two extensively studied examples that lead to functional regression analysis and we provide new insight into the validity of commonly recommended models. The first, concerns the number of eggs laid daily by fruit flies, where the egg-laying trajectory for each fly is usually considered as a curve. Each fly is an individual, the egg-laying trajectory being its response that is modeled as a functional datum. See Chiou \emph{et al.} (2003) for a detailed description of this dataset. Another application will be the Canadian Weather data, intensively studied in Ramsay and Silverman (2005), where the daily rainfall throughout  the year, at each of 35 weather stations is considered as a curve to be explained through a regression model. Many other compelling examples can be found in the monographs of Ramsay and Silverman (2005) and Horv\'{a}th and Kokoszka (2012).

Several models have been designed for regression with functional response. A multiplicative effects model is proposed by Chiou \emph{et al.} (2003) for the egg-laying curves. Concurrent and functional linear models have been analysed by Ramsay and Silverman (2005) for the Canadian Weather dataset. For general functional responses, the functional linear model is the benchmark approach, see Chiou \emph{et al.} (2003), Yao, M\"{u}ller and Wang (2005), Gabrys, Horv\'{a}th and Kokoszka (2010) and the references therein. Recently, alternative nonparametric approaches have been considered; see Ferraty \emph{et al.} (2011), Lian (2011), Ferraty, Van Keilegom and Vieu (2012).

Checking for the goodness-of-fit of a model is an important step in the regression analysis. In a functional regression setup, aspects of the goodness-of-fit seem to be still open. In this paper a new method is proposed to check for the effect of a certain predictor covariate (or predictor) on a functional response. The statistical problem  is to build a  test of the null hypothesis \begin{equation}
H_{0} :\,\, \mathbb{E} \left( U | X \right) = 0 \quad
\mbox{
\rm almost surely (a.s.)} , \label{hh0}
\end{equation}
against the nonparametric alternative $\mathbb{P} [\mathbb{E} \left( U | X \right) = 0]<1$, where $U$ is a functional response, taking values in a separable Hilbert space $\mathcal{H}_1$, and $X$ being the covariate. Our method easily extends to the test of the goodness-of-fit of a regression model with functional response. To this end, it suffices to consider $U,$ the error term, in the mean regression model. In applications, the sample of $U$ is replaced by the residuals obtained from the model fit. \color{black} The predictor $X$ will be allowed to be an univariate or a multivariate or a functional variable taking values in a separable Hilbert space $\mathcal{H}_2$, possibly different from $\mathcal{H}_1$. In the case of egg-laying curves, the predictor will be the total number of eggs, which is a univariate random variable, while in the Canadian Weather dataset, the predictor will be the curve of the daily temperature throughout the year, at each weather station. When little is known about the structure of the data, it is preferable to allow for general alternatives when testing the goodness-of-fit. Moreover, when the link between the response and the predictor is modeled through a nonparametric approach, one should first check whether the predictor has an effect on the response or not.

To the best of our knowledge, only the contributions of Chiou and M\"{u}ller (2007) and Kokoszka \emph{et al.} (2008) have investigated the problem of goodness-of-fit with functional responses. Chiou and M\"{u}ller (2007) introduced diagnostics for the functional regression fit using plots of functional principal component (FPC) scores of the response and the covariate. They also used
residual versus fitted value FPC scores plots. (The FPC scores are the random coefficients in the Karhunen-Lo\`eve expansions.) It is easy to understand that such two-dimensional plots could not capture all types of effects of the covariate on the response, such as for instance, the effect of the interactions of the covariate FPC.
Kokoszka \emph{et al.} (2008) used the response and covariate FPC scores to build a test statistic with an $\chi^2$ distribution under the null hypothesis of the lack of dependence in the functional linear model. Again, by construction, the test of Kokoszka \emph{et al.} cannot detect any nonlinear alternative.

The goodness-of-fit or no-effect against \emph{nonparametric alternatives}, has been  very little explored in the functional data context. In the case of scalar response, Delsol, Ferraty and Vieu (2011) proposed a testing procedure adapted from the approach of H\"{a}rdle and Mammen (1993). Their procedure involves smoothing in the functional space and requires quite restrictive conditions. Patilea, S\'anchez-Sellero and Saumard (2012) and Garc\'{i}a-Portugu\'{e}s, Gonz\'{a}lez-Manteiga and Febrero-Bande (2012) have proposed alternative, nonparametric, goodness-of-fit tests for scalar response and functional covariate using  projections of the covariate. Such projection-based methods are less restrictive and perform well in applications. To the best of our  knowledge, no nonparametric statistical test of no-effect or goodness-of-fit is available when the response is functional.

The paper is organized as follows. In section \ref{section1} the proposed test for the univariate predictor is presented. The test is based on nearest neighbor smoothing. In section \ref{section2}, the test is extended to functional predictors, with a methodology based on projections, as in the spirit of Patilea, S\'anchez-Sellero and Saumard (2012). The functional predictor has to be decomposed in a Hilbert space basis, but we show that our results allow for a data-driven basis, for instance, that given by the functional principal component analysis. A wild bootstrap procedure for computing the critical values with finite samples is also proposed.
In section \ref{section3}, the proposed test is evaluated for simulated data, and used to test the goodness-of-fit of well-known models for the two real data examples mentioned above: the egg-laying trajectories of fruit-flies, and the Canadian Weather dataset.  We conclude that it performs well in applications and, in practice, is useful for model selection.  The proofs are postponed to the appendix.

\section{The  univariate predictor case}\label{section1}

For a clearer presentation and since this case is  important in its own right, we first consider the particular case of a univariate covariate (predictor) $X$.
Our approach seems to be the first one able to check the goodness-of-fit against general alternatives even in this particular case.
Given a sample $(U_1, X_1),\ldots,(U_n, X_n)$ from $(U, X)$, let
\begin{equation*}
Q_{n}=\frac{1}{n(n-1)}\sum\limits_{1\leq i\neq
j\leq n}\langle U_{i} , U_{j} \rangle \frac{1}{h}
K_{h}\left( F_n(X_{i})- F_n(X_{j}) \right),
\end{equation*}
where $\langle U_{i} , U_{j} \rangle$ is the inner product of $U_i$ and $U_j$ taking values in the Hilbert space $\mathcal{H}_1$, $K_{h}\left( \cdot \right) =K\left( \cdot /h\right) $,  $K(\cdot )$ is a kernel,  $h$ is the bandwidth, and $F_n$ is the empirical distribution function of the sample $X_1,\ldots,X_n$.

The statistic $Q_{n}$ is related to statistics considered by Fan and Li (1996) and Zheng (1996) for checks of parametric regressions for finite dimensional data. The main idea here is to replace the common products of responses (or residuals) for univariate response, by the inner products of the functional responses (or residuals).
While Fan and Li (1996) and Zheng (1996) used Nadaraya-Watson weights, symmetrized nearest neighbor weights, introduced by Yang (1981) and Stute (1984), are employed here.
It is well known that, in contrast to the Nadaraya-Watson estimator, the asymptotic variance of the nearest neighbor kernel estimator does not depend on the density of the covariate. This presents some advantages, especially in the case of a functional predictor that will be considered in the following.
Hence, our new statistic is more in the spirit of that introduced by Stute and Gonz\'alez-Manteiga (1996), to test simple linear models with scalar outcome and covariate and homoscedastic error term. Herein we allow for heteroscedasticity of unknown form and hence, in the particular case where $U$ and $X$ are scalar, we extend the framework of Stute and Gonz\'alez-Manteiga (1996).

Let $\|\cdot \|$ denote the norm associated with the inner product $\langle\cdot,\cdot\rangle$ and let us define
\begin{equation*}
Q=\mathbb{E}[\langle U, \mathbb{E}\{U\mid F(X) \}\rangle]= \mathbb{E}[\left\|\mathbb{E}\{U\mid F(X) \}\right\|^2],
\end{equation*}
which by construction is nonnegative.
The idea of the test comes from the fact that $H_0$ is true if and only if $Q=0.$
Then, it will follow from the theoretical results proven for the general covariate case that under $H_{0},$ the quantity $nh^{1/2}Q_{n}$ suitably standardized has asymptotic standard normal distribution. Meanwhile, when $H_0$ is not true  $Q_n$ will converge to a strictly positive $Q$ and this will guarantee consistency against any departure from $H_0$.
To standardize $Q_n$, the following simple variance estimator could be used
\begin{equation*}
\widehat{v}_{n}^{2}=\frac{2}{n(n-1)h}\sum\limits_{j\neq
i}\langle U_{i},U_{j}\rangle^2
K_{h}^{2}\left( F_{n}(X_{i})- F_{n}(X_{j}) \right).
\end{equation*}
Consequently, the test statistic we consider is
\begin{equation*}
T_{n} = n h^{1/2} \frac{Q_{n}}{\widehat{v}_{n}} \;
\end{equation*}
and the associated test of the asymptotic level $\alpha,$  is given by $\mathbb{I} \left(T_{n}\geq z_{1-a} \right) $, where $z_{1-a}$ is the $(1-a)$-th quantile of the standard normal distribution.

\section{The functional predictor case}\label{section2}

The approach introduced in the previous section is based on univariate smoothing and could not be immediately extended to multivariate or functional covariate $X.$ To reduce the case of \color{black}high-dimensional predictor \color{black} to univariate predictor, we use a new dimension reduction idea inspired by Lavergne and Patilea (2008).

\subsection{A dimension reduction lemma}\label{subsection1}

To simplify the presentation and without loss of generality, hereafter we focus on the case where the Hilbert spaces $\mathcal{H}_1$ and $\mathcal{H}_2$ are equal to $L^2[0,1],$ the space of square-integrable functions defined on $[0,1].$
Let
$\langle \cdot, \cdot \rangle$ denote the inner product in $L^2[0,1]$, that is
$$
\langle W_1, W_2 \rangle = \int_0^1 W_1(t) W_2(t) dt,\qquad \forall \; W_1,W_2 \in L^2[0,1].
$$
Let $\|\cdot\|$ be the associated norm. For the moment, let $\mathcal{R} = \{ \psi_1,\psi_2, \cdots \}$  be an arbitrarily  fixed orthonormal basis of the function space $L^2[0,1]$. The extension to a data-driven basis  is considered in section \ref{cesar_is_the_best3}.
Also without loss of generality, hereafter we suppose that $\mathbb{E}[X(t)]=0,$ $\forall t.$
Then the response and the predictor processes  can be expanded to give
\begin{equation}\label{basis_dec}
U(t) = \sum_{j=1}^\infty \langle  U, \psi_j \rangle \psi_j(t) \quad \text{and } \quad X(t) = \sum_{j=1}^\infty \langle  X, \psi_j \rangle \psi_j(t),\qquad t\in[0,1].
\end{equation}
For any integer $p\geq 1$  and any non random $\gamma = (\gamma_1, \cdots,\gamma_p) \in\mathbb{R}^p,$ let $$\langle X, \gamma \rangle = \sum_{i=1}^p \langle  X, \psi_j \rangle \gamma_j$$
and let $F_\gamma$ denote the distribution function (d.f.) of the real-valued variable $\langle X,\gamma \rangle$, \emph{i.e.,} $F_\gamma (y) = \mathbb{P}(\langle X,\gamma \rangle \leq y ),$ $\forall y\in\mathbb{R}.$
Moreover, let $\mathcal{S}^p =\{ \gamma\in\mathbb{R}^p: \|\gamma\|=1\}$ denote the unit hypersphere in $\mathbb{R}^p$.
Our approach relies on the following  extension of Lemma 2.1 of Lavergne and Patilea (2008) to Hilbert space-valued random variables.

\begin{lem}
\label{lem1} Let $U,X\in L^2[0,1]$  be random
functions. Assume that  $\mathbb{E} \| U \| <\infty $ and $\mathbb{E} ( U ) =0.$

(A)
The following statements are equivalent:

\begin{enumerate}
\item $\mathbb{E}(U \mid X)=0$ a.s.
\item $\mathbb{E}\left[\langle U, \mathbb{E}\left(U \mid \langle X,\gamma \rangle \right) \rangle \right]=0$ a.s. $
 \forall p \geq 1 , \forall \gamma \in\mathcal{S}^p.$
\item $\mathbb{E}\left[\langle U, \mathbb{E}\left\{U \mid F_{\gamma}(\langle X,\gamma \rangle) \right\} \rangle \right]=0$ a.s. $
 \forall p \geq 1 , \forall \gamma \in\mathcal{S}^p.$
\end{enumerate}

(B) Suppose, also, that there exists  $s>0$ such that
\begin{equation}\label{cond_x_cond}
\mathbb{E}(\|U\|\exp\{ s \|X\| \}) <\infty.
\end{equation}
If $\mathbb{P} [\mathbb{E} (U\mid X) = 0] <1$, there then exists an  integer $p_0>0$ such that   $\forall p\geq p_0$, the set
$$
\mathcal{A}_p = \{\gamma\in\mathcal{S}^p : \mathbb{E}(U \mid \langle X, \gamma \rangle )=0 \,\, a.s.\, \}
= \{\gamma\in\mathcal{S}^p : \mathbb{E}(U \mid F_\gamma (\langle X, \gamma \rangle) )=0 \,\, a.s.\, \}
$$
has Lebesgue measure zero on the unit hypersphere  $\mathcal{S}^p$  and is not dense.
\end{lem}

Point (A) is a cornerstone for proving the behavior of our test under the null and the alternative hypotheses.  Point (B) shows that in applications it is rather easy  to find directions $\gamma$ able to reveal the failure of the null hypothesis (\ref{hh0}) since, under mild conditions, such directions represent almost all the points on the unit hyperspheres $\mathcal{S}^p$, provided $p$ is sufficiently large. By the Cauchy-Schwarz inequality, condition (\ref{cond_x_cond}) holds if $\mathbb{E}(\|U\|^c)<\infty$ for some $c>1$ and $\mathbb{E}(\exp\{ \varrho \|X\| \}) <\infty$ for $ \varrho = s c/(c-1).$ The exponential moment condition on $\|X\|$ is satisfied in many situations, for instance when  $X$ is a mean-zero Gaussian process with  $\sup_{t\in[0,1]} \mathbb{E} [X^2(t)]<\infty .$  Moreover, in general, moment restrictions on the covariate are not restrictive for goodness-of-fit testing purposes. Indeed, if $X$ does not satisfy condition (\ref{cond_x_cond}), it suffices to transform  $X$ into some variable  $W\in L^2[0,1]$ that generates the same  $\sigma-$field and satisfies (\ref{cond_x_cond}).

Lemma \ref{lem1} contains the case of a multivariate, finite dimension covariate $X$ as a particular case. It will be clear from the following how the testing procedure could be adapted to this situation and hence we focus on the more general case of a functional $X.$

Let
\begin{equation*}
Q(\gamma )=\mathbb{E}[\langle U\;, \mathbb{E}\{U\mid F_{\gamma} (\langle X,
\gamma\rangle) \}\rangle]
\end{equation*}
The following new formulation of $H_{0}$ is a direct consequences of Lemma \ref{lem1} above.

\begin{cor}
\label{charac} Consider a $L^2[0,1]-$valued random variable  $U$
such that  $\mathbb{E}\| U\| <\infty $.
The following statements are equivalent:
\begin{enumerate}
\item The null hypothesis $H_0$ in (\ref{hh0}) holds true.

\item\label{point_33}  $\forall p\geq 1$ and any set $B_p\subset\mathcal{S}^p$ with a strictly positive Lebesgue measure on $\mathcal{S}^p,$
\begin{equation}\label{test_a}
\max_{\gamma\in B_p}Q(\gamma )=0.
\end{equation}
\end{enumerate}
\end{cor}



\label{sec3}

\setcounter{equation}{0}

\subsection{The test statistic with functional predictor}

In view of equation (\ref{test_a}), the goal is to estimate $Q(\gamma).$ Given a  sample  of $(U,X)$, let
\begin{equation*}
Q_{n}\left(\gamma \right) =\frac{1}{n(n-1)}\sum\limits_{1\leq i\neq
j\leq n}\langle U_{i} , U_{j} \rangle \frac{1}{h}
K_{h}\left( F_{\gamma,n}(\langle X_{i},\gamma\rangle)- F_{\gamma,n}(\langle X_{j},\gamma\rangle) \right),\quad \gamma\in\mathcal{S}^p,
\end{equation*}%
where  $K_{h}\left( \cdot \right) =K\left( \cdot /h\right) $,  $K(\cdot )$
is a kernel,  $h$ the bandwidth, and $F_{\gamma,n}$ is the empirical d.f. of the sample $\langle X_1,\gamma\rangle,\cdots,\langle X_n,\gamma\rangle$. \color{black}Ties in the values $\langle X_{i},\gamma\rangle,$ $1\leq i\leq n,$ could be broken by comparing indices, that is if $\langle X_{i},\gamma\rangle = \langle X_{j},\gamma\rangle,$ then we define $F_{\gamma,n}(\langle X_{i},\gamma\rangle)< F_{\gamma,n}(\langle X_{j},\gamma\rangle)$ if $i<j.$ However, for simplicity in our assumptions below we will assume that the $\langle X_{i},\gamma\rangle' $s have a continuous distribution for all $\gamma.$ \color{black}

The statistic $Q_{n}\left(\gamma \right)$ is related to the statistic considered by Patilea, S\'anchez-Sellero and Saumard (2012) who used a Nadaraya-Watson regression estimator instead of the  nearest neighbor (NN) approach. \color{black}
Since the asymptotic variance of the NN kernel estimator does not depend on the density of the covariate, in our case the covariate is $\langle X,\gamma\rangle,$ one could more confidently use the same bandwidths $h$ for any $\gamma$ to define $Q_n(\gamma).$
\color{black}
The projections of the covariates were also considered by Lavergne and Patilea (2008); see also Cuesta-Albertos \emph{et al.} (2007), Cuesta-Albertos, Fraiman and Ransford (2007). The extension of the scope  to functional responses seems to be new. As in the univariate predictor case,  we allow for heteroscedasticity of unknown form.

Under $H_{0}$, by the Central Limit Theorem (CLT) for degenerate $U-$statistics,
for  fixed $p$ and $\gamma\in\mathcal{S}^p$,  $nh^{1/2}Q_{n}\left(\gamma \right)$ has an
asymptotic centered normal distribution.
Here we use the CLT in Theorem 5.1 in de Jong (1987).
We will show de Jong CLT still applies and the asymptotic normal distribution is preserved even when $p$ grows at a suitable rate with the sample size. On the other hand, Lemma \ref{lem1}-(B) indicates that if $p$ is sufficiently large, the maximum of $Q\left(\gamma\right)$ over $\gamma$  stays away from zero under the alternative hypothesis and this will guarantee consistency against any departure from $H_{0}$.

The statistic $Q_{n}(\gamma)$ is expected to be close to $Q(\gamma  )$ uniformly in $\gamma$, provided $p$ increases suitably. Then a natural idea would be to build a test statistic using the maximum of $Q_{n}(\gamma)$ with respect to $\gamma$. However, as in the finite dimension covariate case, under $H_0$ one expects $Q_{n}(\gamma)$ to converge to zero for any $p$ and $\gamma$ and thus the objective function of the maximization problem to be flat.
Therefore we will choose a direction $\gamma$ as the least favorable direction for the null hypothesis $H_0$ obtained from a penalized criterion based on a standardized version of $Q_{n}\left(\gamma \right)$; see also Lavergne and Patilea (2008) for related approaches.
More precisely, let us fix some infinite-dimensional vector $(b_{01},b_{02},\cdots)\in \mathbb{R}^\infty$  with $\sum_{j=1}^\infty b_{0j}^2 <\infty.$ Such a vector could be interpreted as an initial \emph{guess} of an unfavorable direction for $H_0.$
For any given $p\geq 1$ such that $\sum_{j=1}^p b_{0j}^2 >0$, let
$$
\gamma_0^{(p)} = \frac{(b_{01},\cdots,b_{0p})}{\| (b_{01},\cdots,b_{0p})\|} \in\mathcal{S}^p\;,
$$
where here $\|\cdot \|$ denotes the norm in $\mathbb{R}^p$.

Let
\begin{equation}
\widehat{v}_{n}^{2}( \gamma) =\frac{2}{n(n-1)h}\sum\limits_{j\neq
i}\langle U_{i},U_{j}\rangle^2
K_{h}^{2}\left( F_{\gamma,n}(\langle X_{i},\gamma\rangle)- F_{\gamma,n}(\langle X_{j},\gamma\rangle) \right),
\label{var_est_v_n}
\end{equation}
be an estimate of the
variance of $nh^{1/2}Q_{n}(\gamma) ,$ $\gamma\in\mathcal{S}^p$. Given  $B_p\subset \mathcal{S}^p$ with positive Lebesgue measure in $\mathcal{S}^p$ and which contains $\gamma_0^{(p)}$,  the least favorable direction $\gamma$ for $H_0$ is defined by
\begin{equation}  \label{bet}
\widehat{\gamma }_{n} = \arg \max_{\gamma\in B_p} \left[ n h
^{1/2} Q_{n}(\gamma)/\widehat{v}_{n}(\gamma) - \alpha _{n} \mathbb{I}_{\left\{
\gamma \neq \gamma_0^{(p)}  \right\}} \right] \;,
\end{equation}
where $\mathbb{I}_A$ is the indicator function of a set $A$,  and $\alpha_n$, $n\geq 1$ is a sequence of positive real numbers decreasing to zero at an appropriate rate, which depends on the rates of $h$ and $p$ and will be made explicit below. Using a standardized version of $Q_{n}(\gamma)$ avoids scaling $\alpha_n$ according to the variability of the observations.
Let us note that the maximization used to define $\widehat{\gamma }_{n}\in B_p \subset \mathcal{S}^p$ is a finite dimension optimization problem. The choice of $\gamma_0^{(p)}$ will be shown to be theoretically irrelevant. It will not affect the asymptotic critical values and the consistency results. Practical aspects related to the choice of $\gamma_0^{(p)}$ and $B_p $ will be discussed in section \ref{cesar_is_the_best3}.

We will prove that with  suitable rates of increase for $\alpha_n$ and $p$ and decrease for $h$, the probability of the event $\{
\widehat{\gamma}_{n} = \gamma_{0}^{(p)}\}$ tends to 1 under $H_{0}$.
Hence $Q_{n}(\widehat{\gamma}_{n})/\widehat{v}_{n}(\widehat{\gamma})$ behaves asymptotically as $
Q_{n}(\gamma_{0}^{(p)})/\widehat{v}_{n}(\gamma_0^{(p)})$, even when $p$ grows with the sample size.
Therefore the test statistic we consider is
\begin{equation}\label{test_stat}
T_{n} = n h^{1/2} \frac{Q_{n}(\widehat{\gamma}_{n})}{
\widehat{ v}_{n} (\widehat{\gamma}_{n})} \; .
\end{equation}
We will show that an asymptotic $a$-level test is given by $\mathbb{I} \left(
T_{n}\geq z_{1-a} \right) .$

\subsection{Behavior under the null hypothesis}\label{beh_null_hyp}
\setcounter{hp}{3}

\begin{hp}
\label{D}
\hspace{-0.1cm}
\begin{enumerate}
\item[(a)]
The random vectors $(U_{1},X_{1})
,\ldots ,(U_{n},X_{n})$ are independent draws from the random vector $(U,X)\in L^2[0,1] \times L^2[0,1]$ that satisfies  $\mathbb{E}\| U \|^{8}<\infty ;$ moreover, $\exists \varrho >0$ such that $\mathbb{E}[\exp(\varrho\|X\|)]<\infty.$

\item[(b)] For any $p \geq 1$ and any $\gamma \in \mathcal{S}^p$, the d.f. $F_\gamma $ is continuous.

\item[(c)]
$\exists \ \underline{\sigma}^{2},$ $C_1, C_2 >0$ and $\nu >2$
such that:
\begin{enumerate}
\item[(i)] $0 < \underline{\sigma}^{2} \leq \mathbb{E}(\langle U_{1},U_{2}\rangle^2 \mathbb{I}_{\{ |\langle U_{1},U_{2}\rangle |\leq C_1 \}} \mid X_{1},X_{2})$ almost surely;
\item[(ii)] $ \mathbb{E}\left[\Vert U\Vert^{\nu} \mid X \right] \leq C_2$.
\end{enumerate}

\item[(d)] For any $p \geq 1$,  $ \gamma_0^{(p)} \in B_p\subset\mathcal{S}^p,$  $B_p$ are open subsets of $\mathcal{S}^p$ and  $B_p\times 0_{p^{\,\prime} - p} \subset B_{p^{\,\prime}},$ $\forall 1\leq p< p^{\,\prime}$ where $0_p\in\mathbb{R}^p$
denotes the null vector of dimension $p$.
\label{ass_dc}
\end{enumerate}
\end{hp}

The continuity condition  in Assumption \ref{D}-(b) is a mild assumption that simplifies the NN smoothing.
Assumption \ref{D}-(c)  serves to prove that the variance of $Q_n(\cdot)$ is uniformly bounded away from zero and infinity. The mild conditions  on  $B_p$ simplify the proofs for the consistency and   are satisfied, for instance, when $B_p$ is a half unit hypersphere.

\setcounter{hp}{10}
\begin{hp}
\label{K} \hspace{-0.1cm}
\begin{enumerate}
\item[(a)] The kernel $K$ is a continuous  density on the real line such that $K(x) = K(-x)$ and $K(\cdot)$ is non increasing on $[0,\infty)$. Moreover the Fourier Transform of $K$ is integrable.
\item[(b)] $h\rightarrow 0$ and $nh^{2} \rightarrow
\infty.$

\item[(c)] $p\geq 1$ increases to infinity with $n$ and there exists a constant $\lambda>0$ such that $p\ln^{-\lambda} n $ is  bounded.
\end{enumerate}
\end{hp}

The first step for deriving a test statistic is the study of the behavior of the process $Q_{n}(\gamma),$ $\gamma\in B_p$, under $H_{0}$ when $p$  increases with the sample size. This study is greatly simplified by the fact that for a fixed $n,$ up to permutations of lines and/or columns, the matrix with entries
$K_h (F_{\gamma,n} (\langle X_i,\gamma\rangle) - F_{\gamma,n} (\langle X_j,\gamma\rangle)),$ $1\leq i,j\leq n,$ is equal to that with entries matrix $K_h ((i-j)/nh),$ $1\leq i,j\leq n,$  for any dimension $p$ and direction $\gamma$.

\begin{lem}
\label{leem1}  Under  Assumptions \ref{D} and \ref{K} and if $H_0$ holds true,
$$\sup_{\gamma \in B_p\subset\mathcal{S}^p}|Q_{n}(\gamma)|= O_{\mathbb{P}
}(n^{-1}h^{-1/2} p \ln n) .$$ Moreover, if $\widehat v_n^2 (\gamma)$ is the estimate defined in equation (\ref{var_est_v_n}),
\begin{equation*}
\sup_{\gamma \in B_p\subset\mathcal{S}^p} \{1/\widehat v_n^2 (\gamma)\} = O_{\mathbb{P}}(1).
\end{equation*}
\end{lem}

We next describe the behavior of $\widehat{\gamma }_{n}$ under $H_{0}$. A suitable rate $\alpha_n$ will make $\widehat{ \gamma}_{n}$ to be equal to $\gamma_{0}^{(p)}$ with high probability. Under the null, $\alpha_n$ has to grow to infinity sufficiently fast to render the probability of the event $\{ \widehat{ \gamma}_{n} = \gamma_{0}^{(p)} \}$ close to 1. We will see below that, for better detection of  the alternative hypothesis,  $\alpha_n$ should grow as slowly as possible. Indeed, slower rates for $\alpha_n$ will allow directions $\hat\gamma_n$  to be selected which could be better suited than $\gamma_0^{(p)}$ for revealing the departure from the null hypothesis. The rate of $p$ is also involved in the search of a trade-off for the rate of $\alpha_n$: a larger $p$ renders  the rate of uniform convergence to zero of $Q_n (\gamma)$, $\gamma\in B_p$ slower, and hence requires a larger $\alpha_n$.

\begin{lem}
\label{beta} Under Assumptions \ref{D}, \ref{K}, for a positive sequence
$\alpha _{n} $, $n\geq 1$ such that $\alpha _{n}p^{-1}\ln^{-1} n \rightarrow
\infty$,  $$\mathbb{P}(\widehat{ \gamma}_{n} = \gamma_{0}^{(p)})
\rightarrow 1, \quad \text{ under } H_{0}.$$
\end{lem}

The following result shows that the asymptotic critical values of our test statistic are standard normal.

\begin{theor}
\label{as_law} Under the conditions of Lemma \ref{beta} and if the hypothesis $H_0$ in (\ref{hh0}) holds true, the  test statistic $
T_{n} $ converges in law to a standard normal. Consequently, the test given by $\mathbb{I}(T_n \geq z_{1- a})$, with $z_a$  the $(1- a)-$quantile of the standard normal distribution,  has an asymptotic level  $ a.$
\end{theor}

Theorem \ref{as_law} could be derived in the case of a finite dimension  covariate $X$ under Assumption \ref{D}-(a,b,c)  and Assumption \ref{K}-(a,b). Since no dimensional reduction is required in the univariate case, no exponential moment condition  is required when $X$ is univariate.

Under technical conditions ensuring that the sample of $U$ is estimated sufficiently accurately, the test statistic $T_n$ will still have standard normal critical values when the $U_i$'s are replaced by some estimates. Patilea, S\'anchez-Sellero and Saumard (2012) provide complete arguments for their test  in the case where the $U_i$'s are the residuals of the functional linear model with scalar responses. Similar arguments could be used with functional responses.
To keep this paper to a  reasonable length, the theoretical investigation of the extension to the case of estimated responses $U_i$ will be omitted. However, some empirical evidence from extensive simulation experiments are reported in section \ref{section3}.

\subsection{The behavior under the alternatives}

Our test is consistent against the general alternative $$H_1:\; \mathbb{P}[\mathbb{E}(U\mid X)=0] <1,$$
\emph{i.e.,} the probability that the test statistic $T_n$ is larger than any quantile $z_{1- a}$ tends to one under $H_1.$ This could be rapidly understood from the following
simple inequalities:
\begin{equation}
T_{n} \geq \max_{\gamma\in B_p} \frac{nh^{1/2} Q_{n}(\gamma)}{\widehat{v}_{n}(\gamma) } -\alpha _{n}
\geq \frac{nh^{1/2} Q_{n}(\widetilde\gamma ) }{\widehat{v}_{n}(\widetilde\gamma)} -\alpha_{n}, \quad \forall \widetilde \gamma\in B_p\subset\mathcal{S}^p,
 \label{eqaa}
\end{equation}
with $\widehat{v}_{n}(\gamma)$  defined in  (\ref{var_est_v_n}). Since $Var(\langle U_{1},U_{2}\rangle \mid  X_{1},X_{2}) \geq \underline{\sigma}^2 $, it is clear that
$1/\widehat{v}_{n}(\widetilde \gamma)=O_{\mathbb{P}}(1)$ for all $\widetilde \gamma$. On the other hand, from Lemma \ref{lem1}, there exists a $p_0$ and a $\widetilde \gamma\in B_{p_0}$ such that the expectation of $ Q_{n}(\widetilde \gamma )$ does not approach zero as the sample size grows to infinity and $h$ decreases to zero. On the other hand, for any $p>p_0$ and any $n$ and $h$, clearly $\max_{\gamma \in B_p} Q_{n}(\gamma ) \geq Q_{n}(\widetilde \gamma )$, because $B_{p_0} \times 0_{p-p_0} \subset B_{p}$. All these facts show why our test is an omnibus test, that is consistent against nonparametric alternatives, provided that $p\rightarrow \infty.$

To formally state the consistency result, let $\delta(X)$ be some $L^{2}[0,1]$-valued function such that $\mathbb{E}[\delta(X)] = 0$ and $0<\mathbb{E}[\Vert \delta(X) \Vert^{4}]<\infty$,  and let  $r_n,$ $n\geq 1$  be a sequence of real numbers that either decrease to zero or  $r_n\equiv 1.$  Consider the sequence of alternative hypotheses:
\begin{equation*}
H_{1n}:\; U = U^0 + r_n \delta(X),\quad  n\geq 1,  \quad\text{with} \;\; U^0 \in L^2[0,1],\;\; \mathbb{E}(U^0\mid X) = 0.
\end{equation*}
We show below that such directional alternatives can be detected \color{black}as long as \color{black} $r_n^2 n h^{1/2}  / \alpha_n \rightarrow \infty.$ This is exactly the condition one would obtain with scalar covariate; see Lavergne and Patilea (2008).
However, in the functional data framework, to obtain the convenient standard normal critical values,  we need   $1/\alpha_n =o(p^{-1}\ln^{-1} n)$. Hence, the rate $r_n$ at which the alternatives $H_{1n}$ tend to the null hypothesis should satisfy  $r_n^{2}n h^{1/2}/\{p\ln n\}\rightarrow \infty $.

\begin{theor}
\label{altern}
Suppose that
\begin{enumerate}
\item[(a)] Assumption \ref{D}  holds true with $U$ replaced by $U^0$;
 \item[(b)] Assumption \ref{K} is satisfied and in addition $nh^4 \rightarrow \infty$ and  there exists a constant $C$ such that
 $|K(u) - K(v)|\leq C|u-v|,$ $\forall u,v\in\mathbb{R};$
 \item[(c)] $\alpha _{n}/\{ p \ln n \} \rightarrow
\infty$ \color{black}and
$r_n $, $n\geq 1$ such that \color{black} $r_n^2 n h^{1/2}  / \alpha_n\rightarrow \infty$;
  \item[(d)] $\mathbb{E}[\delta(X)] = 0$ and $0<\mathbb{E}[\Vert \delta(X) \Vert^{8}]<\infty;$

  \item[(e)] there exists $ p$  and $\widetilde \gamma\in B_{p}\subset \mathcal{S}^{p}$  (independent of $n$) such that $\mathbb{E}[\delta(X) \mid \langle X,\widetilde  \gamma\rangle]\neq 0$ and,  $\forall t\in[0,1]$, the Fourier Transform of $\overline \delta (t,\cdot) = \mathbb{E}[\delta(X)(t)\mid F_{\widetilde\gamma} (\langle X, \widetilde\gamma\rangle ) = \cdot ]$ is integrable.

\end{enumerate}
Then the test based on $T_n$ is consistent against the sequence of alternatives $H_{1n}.$
\end{theor}

\bigskip

The additional Lipschitz condition on the kernel $K(\cdot)$  and the restriction on the bandwidth range in Theorem \ref{altern}-(b) are reasonable technical conditions that simplify the proof of consistency. The zero mean condition for $\delta(\cdot)$ keeps the mean of $U$ equal to zero under the alternative hypotheses $H_{1n}$. The existence of vectors $\widetilde \gamma$ with $\mathbb{E}[\delta(X) \mid \langle X,\widetilde  \gamma\rangle]\neq 0$ is guaranteed by Lemma \ref{lem1}-(B). In
Theorem \ref{altern}-(e) we impose a convenient mild technical condition on one of such vectors. Finally, Theorem \ref{altern} could be easily adapted to the case of a \color{black}finite dimension covariate. \color{black} The details are omitted.

\subsection{Practical aspects}\label{cesar_is_the_best3}

In the case of a functional covariate, the goodness-of-fit procedure we propose in this paper requires the choice of several quantities: the orthonormal basis $\mathcal{R}$ in the space of $X$, the order $p,$ the penalty amplitude $\alpha_n$, the privileged direction $\gamma_0^{(p)},$ the set $B_p$ and the bandwidth $h$. \color{black}In this section we provide some guidelines on how these quantities could be chosen by the practitioner, except for $\gamma_0^{(p)},$  $B_p$ and $h$ for which the choice will be discussed in the Supplementary Material. \color{black} Before doing this, let us point out that the choice of the basis in the space of $U$ is not really an issue. In applications, the statistician only has to compute the $n(n-1)$  products $\langle U_i, U_j \rangle$ and this could be easily done  with high accuracy and low computational costs in any basis.

Our theoretical results above are derived for a fixed basis $\mathcal{R}$ in the space of $X$. The assumptions used to derive these results impose only very mild conditions on the basis $\mathcal{R}$, see Assumption \ref{D}-(b) and condition (e) in Theorem \ref{altern}. However, the choice of the basis could influence the finite sample performances of the test. Clearly, the practitioner would prefer a basis that allows for an accurate \color{black}low-dimensional representation \color{black} of the covariate and hence for a low $p$ in our testing procedure. A widely used basis is that given by the eigenfunctions of the covariance operator $\Gamma$ of $X$ that is defined by:
$$
(\Gamma v)(t) = \int \sigma(t,s) v(s) ds ,\qquad v\in L^2[0,1],
$$
where $X$ is supposed to satisfy the condition $\int \mathbb{E}(X^2(t)) dt < \infty$ and $\sigma(t,s) = \mathbb{E}[ \{X(t) - \mathbb{E}(X(t))  \}  \{ X(s) - \mathbb{E}(X(s)) \}] $ is supposed positive definite. Let $\lambda_1\geq \lambda_2 \geq \cdots$ denote the ordered eigenvalues of $\Gamma$ and let $\mathcal{R}=\{\psi_1,\psi_2,\cdots\}$ be the corresponding basis of eigenfunctions  of $\Gamma$ that are usually  called the functional principal components (FPC). The FPCs represent the orthonormal basis of the Karhunen-Lo\`eve decomposition of $X$ and provide optimal  low-dimensional representations  of $X,$ with respect to the mean-squared error. See, for instance, Ramsay and Silverman (2005).
In some cases where the law of $X$ is given, the FPCs are available. However, most of the time this is not the case and the FPCs have to be estimated from the empirical covariance operator
$$
(\widehat \Gamma v)(t) = \int \widehat \sigma(t,s) v(s) ds ,
$$
where $\widehat \sigma(t,s) = n^{-1} \sum_{i=1}^n  \{X_i(t) - \overline X_n (t)  \}  \{ X_i(s) - \overline X_n (s) \}] $ and $\overline X_n (t) = n^{-1} \sum_{i=1}^n  X_i(t).$
Let $\widehat \lambda_1\geq \widehat \lambda_2 \geq \cdots\geq 0$ denote the eigenvalues of $\widehat \Gamma$ and let $\widehat \psi_1,\widehat \psi_2,\cdots$ be the corresponding basis of eigenfunctions, \emph{i.e.,} the estimated FPCs. We adopt the usual identification condition and we suppose that for any $j,$ $\langle  \psi_j, \widehat \psi_j\rangle \geq 0$. For any $\gamma = (\gamma_1,\cdots,\gamma_p)\in\mathcal{S}^p $ let
$$
\langle X_i ,\gamma\rangle_n = \sum_{k=1}^p\gamma_k \int_{[0,1]} X_i(t)\widehat\psi_k (t) dt.
$$
Let $\widehat T_n$ be the test statistic obtained from equations (\ref{bet}) and (\ref{test_stat}) after replacing all the inner products $\langle X_i ,\gamma\rangle$ by the estimated versions $\langle X_i ,\gamma\rangle_n.$
Below we show that the test $\mathbb{I}(
\widehat T_{n}\geq z_{1-a} ) $ behaves asymptotically like the test $\mathbb{I} \left(
T_{n}\geq z_{1-a} \right). $ For the behavior under the null hypothesis, no additional assumption is required. For  consistency, we impose  mild conditions on $X$ and a slightly more restrictive bandwidth range.

\begin{cor}\label{FPC}
a) Under the same conditions, the conclusion of Theorem \ref{as_law} remains true if $T_n$ is replaced by $\widehat T_n$.

b) In addition to the conditions of Theorem  \ref{altern} assume that
\begin{enumerate}
\item there exist $C,\eta>0$ such that $\lambda_j - \lambda_{j+1}\geq C j^{-\eta}$, $\forall j\geq 1$;

\item the vector $\widetilde \gamma\in \mathcal{S}^{p}$ in condition (e) of Theorem \ref{altern} is such that the variable $\langle X,\widetilde \gamma\rangle$ has a bounded density $f_{\widetilde \gamma}$ ;

\item $nh^4/p^{2\eta+1}\ln^2 n \rightarrow \infty.$
 \end{enumerate}
The conclusion of Theorem  \ref{altern}, then remains true if $T_n$ is replaced by $\widehat T_n$.
\end{cor}

The  condition on the spacings between the ordered eigenvalues of $\Gamma$
is a   common condition in functional data modeling.
In view of Lemma \ref{lem1}-(B), almost any unit norm vector of finite but sufficiently large dimension is a candidate to be $\widetilde\gamma$. Hence
the bounded density condition for some $\langle X,\widetilde \gamma\rangle$ is also a mild restriction. For instance, it is satisfied for any unit norm vector $\widetilde\gamma \in\mathbb{R}^{p}$ if $X$ is a gaussian process.

The value of $p$ needs to grow to infinity to guarantee consistency against general alternatives. Meanwhile, large $p$ makes the optimization over $B_p$ more difficult.
Using the FPC basis could be a good compromise to detect general alternatives with small $p.$ If the $\lambda_j$'s then decrease as fast as a power of $j,$  an automatic choice for $p$ could be given by $\min\{p : \sup_{j\geq p} \lambda_j \leq C\ln^{-1} n \}$ for some constant $C.$ This would result in a logarithmic rate for $p$. In practice $\lambda_j$ should be replaced by the estimates $\widehat \lambda_j$, but the  rate of $p$ will not change because $\sup_{j\geq 1} |\widehat \lambda_j - \lambda_j |$  is of order $O_{\mathbb{P}}(n^{-1/2})$ under mild conditions; see, for instance, Horv\'{a}th and Kokoszka (2012), chapter 2. In practice simple empirical rules work as well. For instance $p$ could be the smallest value such that more than some fixed high percentage, say 95\%, of the variance within the covariate sample is captured by the first $p$ principal components.


Under the null hypothesis, if $n\rightarrow\infty$ and $p$ increases with $n$ at a suitable rate, the ratio  $n h ^{1/2} Q_{n}(\gamma)/\widehat{v}_{n}(\gamma)$ behaves like a standard normal for any given sequence of $\gamma\in\mathcal{S}^p$.
Meanwhile the supremum of this ratio with respect to $\gamma\in B_p$ diverges in probability with a rate smaller or equal to $p \ln n .$ Hence $\alpha_n$ has to grow to infinity faster than $p \ln n.$ In practice, for sample sizes of hundreds, larger $\alpha_n$ (like for instance $\alpha_n = 10$) will likely result in taking $\widehat \gamma _n = \gamma_0^{(p)}$ and in this case the standard normal critical values will be quite accurate. Having $\widehat \gamma _n = \gamma_0^{(p)}$ might be reasonable when the practitioner judges $\gamma_0^{(p)}$ trustful for detecting alternatives. On the other hand, smaller $\alpha_n$ (for instance $\alpha_n=1$ or 2) will probably lead to a value of the test statistic equal to the maximum value of $n h
^{1/2} Q_{n}(\gamma)/\widehat{v}_{n}(\gamma)$ and hence in general, the test will overreject the null hypothesis. Meanwhile, smaller $\alpha_n$ is preferable for detecting general alternatives. On the basis of our detailed simulation investigations, we recommend  values for $\alpha_n$ between 2 and 5 and a correction of the critical values through resampling, as explained below.

%

Now, let us propose a wild bootstrap procedure that could be used for correcting the finite sample critical values. \color{black}In particular, such a correction is useful to take into account the effect of the penalty $\alpha_n$ with finite samples. \color{black} The bootstrap sample, denoted by $U_i^b$ , $1\leq i\leq n$, is obtained as follows:  $U_i^b = \zeta_i U_i$, $1\leq i\leq n$, where $\zeta_i$, $1\leq i\leq n$ are independent random variables with expectation zero and variance one. In particular, for their common distribution we chose the two-points distribution  proposed by Mammen (1993), that is, $\zeta_i=-(\sqrt{5}-1)/2$ with probability $(\sqrt{5}+1)/(2\sqrt{5})$ and $\zeta_i=(\sqrt{5}+1)/2$ with probability $(\sqrt{5}-1)/(2\sqrt{5})$. As with the original test statistics, a bootstrap test statistic $T_n^b$ is built from a bootstrap sample. Similarly, let $\widehat T_n^b$ be the bootstrap test statistic obtained from this procedure applied with the estimated FPC basis. As usually, for any $a\in(0,1)$, the $(1-a)-$th conditional quantile of $T_n^b$ or $\widehat T_n^b$ given $(U_1,X_1),\cdots, (U_n,X_n)$ could be approximated  using a Monte-Carlo method.
The asymptotic validity of this bootstrap procedure is guaranteed by Theorem \ref{as_law} and the following result.

\begin{theor}\label{bbot}
Under the null hypothesis $H_0$ and if  the conditions of Theorem \ref{as_law}  hold true,
$$
\sup_{x\in\mathbb{R}} \left| \mathbb{P}\left( T_n^b \leq x \mid U_1,X_1,\cdots, U_n,X_n \right) - \Phi(x) \right| \rightarrow 0,\qquad \text{in probability,}
$$
where $\Phi(\cdot)$ is the standard normal distribution function.
Under the sequence of alternative hypotheses $H_{1n}$ and if the conditions of Theorem \ref{altern}  hold true, for any $a\in (0,1),$ $\mathbb{P}\left(T_n >  z^b_{1-a, n}\right)\rightarrow 1,$ where $z^b_{1-a, n}$  is the $(1-a)-$th conditional quantile of the  statistic $T_n^b$ given
$(U_1,X_1),\cdots, (U_n,X_n).$ The statements remain true with $T_n^b$  replaced by $\widehat T_n^b$.
\end{theor}

\bigskip

Finally, the optimization problem $\max_{\gamma\in\mathcal{S}^p} Q_{n}(\gamma)/\widehat{v}_{n}(\gamma)$ can be solved with reasonable computational effort for small $p$ (up to $p=5$) by taking a grid of values in the hypersphere $\mathcal{S}^p$. For a larger $p$, a grid of values is not feasible in terms of computation time. In this case, we propose a sequential algorithm based on successive one-dimensional optimizations. If one considers $p=2,$ the directions $\gamma$ on the hypersphere $\mathcal{S}^2$ can be represented as $\gamma=(\cos\theta, \sin\theta)$ with $\theta\in[0,2\pi)$. As  mentioned, one can make the restriction of $\theta\in[0,\pi)$, since half of the circle is sufficient  to consider all directions in the plane. An equally-spaced grid of values of $\theta$ in $[0,\pi)$ provides an equally-spaced grid of directions in $\mathcal{S}^2$. Next, if $p=3,$ then the first step would be to optimize with respect to  the first two components as before. Let $(\gamma_1^*,\gamma_2^*)$ be such an optimal direction in two dimensions. The next step would be to optimize in the set of directions $\cos\theta \cdot (\gamma_1^*, \gamma_2^*, 0) +\sin\theta \cdot
(0, 0, 1)$ for $\theta\in[0,\pi)$. This is again a one-dimensional optimization that can be solved with a grid of values in the interval $[0, \pi)$. This procedure can be applied to a possible fourth dimension, from the optimal direction obtained with the first three dimensions, and so on until the chosen number of components $p$ is reached. This method would require $(p-1)$ one-dimensional optimizations. Simulations given in section \ref{sect_4_2} were carried out with a sequential algorithm and a grid of 50 points in each one-dimensional optimization. In the Supplementary Material, some empirical results are presented to show that the statistical properties obtained with the sequential algorithm are very close to those obtained with a full-dimensional optimization in $\mathcal{S}^p$. We also show in Lemma B in the Supplementary Material that the sequential search will lead toward a direction able to reveal any departure from the null hypothesis. On the other hand, the asymptotic behavior of the test under the null hypothesis is not affected by the sequential search, since the dominant part of the test statistic will still be given by $nh^{1/2} Q_n(\gamma_0^{(p)})/\widehat v_n(\gamma_0^{(p)}),$
exactly as in the case of a full-dimensional search.



\section{Empirical study}\label{section3}

The proposed methods were applied to simulated data as well as to real data. We first present the results obtained for a univariate predictor. Functional predictor models are considered later.

\subsection{Univariate predictor}

We consider a model where the response, $U_i$, is functional and the predictor, $X_i$, is univariate. Under the null hypothesis, $X_i$ has no effect on $U_i$ and a common curve $\mu(t)$ represents the expectation of $U_i(t)$, that is,
$$U_i(t)=\mu(t)+\epsilon_i(t),\quad 1\leq i\leq n, \qquad \mu(t)=0.01\cdot\exp(-4\cdot(t-0.3)^2), \;t\in[0,1],$$
where $\epsilon_i$ are independent Brownian bridges, also independent of $X_i$. The $X_i$'s have a log-normal distribution with mean $3$ and standard deviation $0.5$. Under the alternative,
$$U_i(t)=\mu(t)\cdot X_i+\epsilon_i(t),\qquad 1\leq i\leq n.$$
This is a multiplicative effects model, as proposed by Chiou, M\"{u}ller and Wang (2004) for the medflies data. Figure \ref{Figure0} represents the curve $\mu(t)$ which is the common curve shape for all individuals in the multiplicative effects model.

\medskip
\begin{center}
\emph{Insert Figure 1 here}
\end{center}
\medskip


The statistic was computed with the Epanechnikov kernel, $K(x)=(3/4)(1-x^2)\mathbb{I}_{\{|x|\leq 1\}}$. Table 1 below shows percentages of rejections for several nominal levels $10\%, 5\%, 1\%$ and sample sizes $n=100,200$, under the null hypothesis and under the alternative, and different values of the bandwidth, $h=c_h\cdot n^{-2/9}.$ The coefficient $c_h$ is indicated in the table. For each original sample, we used 499 bootstrap samples to compute the critical value. One thousand original samples were generated to approximate the percentages of rejection. Each original sample was generated once for all the significance levels and bandwidths.

The level is respected under the null hypothesis, with the approximation being better for larger sample size. Under the alternative, the power is increasing with  sample size, and there is not much effect of the bandwidth.

\medskip
\begin{center}
\emph{Insert Table 1 here}
\end{center}
\medskip


\subsubsection{Application to egg-laying curves of fruit flies}

As briefly explained in the Introduction, Chiou \emph{et al.} (2003) proposed a multiplicative effects model for the egg-laying curve of each fly, which can be expressed
$$U_i(t)=\mu(t)\phi(X_i)+\epsilon_i(t),\qquad 1\leq i\leq n,$$
where $U_i(t)$ is the number of eggs laid by the $i-$th fly on day $t$, $\mu(t)$ represents the common shape of the egg-laying curve for all flies, $\phi(X_i)$ is a multiplicative effect related to $X_i$ that denotes the total number of eggs laid by the $i-$th fly and $\epsilon_i(t)$ is an error term. The data under analysis consists of 936 flies that laid at least one egg in their lifetime. The complete data set is available on the web pages of the authors of Chiou \emph{et al.} (2003).

We applied the new test to check the effect of the total number of eggs on the egg-laying curve. The null hypothesis of no-effect versus nonparametric alternative was clearly rejected, with $p-$values extremely close to zero. The test was next also applied to check the goodness-of-fit of the multiplicative effects model. Note that under the multiplicative effects model, both functions $\mu(t)$ and $\phi(x)$ are nonparametrically estimated. To obtain the goodness-of-fit test, the residuals coming from the adjusted model were used in the expression of the test statistic.
Our test clearly rejects the model ($p-$value less than 0.001). The cause could be due to some discrepancies already found by Chiou \emph{et al.} (2003) in the peak of the egg-laying curve. Some flies showed a peak quite far from the model, and in particular those flies which produce fewer eggs (smaller value of $X_i$) typically had shorter lifetimes and an earlier peak.
When $191$ of these flies with an anomalous peak were deleted from the data set, the remainder sub-sample of $745$ flies provided a better adjustment of the model, which was no longer rejected by our test (the $p-$value was 0.124). The new test was then useful to confirm the anomalies found by Chiou, M\"{u}ller and Wang (2004) in some individuals with respect to the multiplicative effects model. Once these individuals were removed, the model was accepted by the test.

\subsection{Functional predictor}\label{sect_4_2}

We shall now assess the performance of the test in the case of a functional predictor. The sequential algorithm described in Section 3.5 will be used to compute the test statistic with a grid of 50 points in each one-dimensional optimization. In all our simulated models the empirical percentages of rejection will be provided on the basis of one thousand original samples. The critical values for each sample will be approximated by means of 499 bootstrap replicates.

Together with the assessment of the level under the null and the power under the alternative, we shall compare our test with the procedure proposed by Kokoszka \emph{et al.} (2008), which is a parametric test of the functional linear effect. 

The first situation we considered then was a functional linear model given by
\begin{equation} \label{eq:functional-linear}
U_i(t)=\int_0^1 \zeta(s,t) X_i(s)\,ds+\epsilon_i(t),\qquad 1\leq i\leq n
\end{equation}
where $X_i$ and $\epsilon_i$ are independent Brownian bridges and $\psi$ is square-integrable over $[0,1]\times[0,1]$. The kernel $\psi$ was chosen to be $\zeta(s,t)=c\cdot \exp(t^2+s^2)/2$, with $c=0$ under the null and $c\neq 0$ under the alternative.

The estimated functional principal components of the covariate are used as the basis. Different possibilities for the privileged direction $\gamma_0^{(p)}$ were considered. We present here the results for an uninformative direction, with the same coefficients in all basic elements. For the penalization we used the value $\alpha_n=2$, which provides a good trade-off between the privileged direction and the direction maximizing the standardized statistic.
The Epanechnikov kernel was again used to compute the statistic in each direction. The bandwidth was chosen following the rule $h=n^{-2/9}$.

Table 2 shows the empirical powers obtained for different significance levels, and sample sizes, with $c=0$ representing the null hypothesis and $c=0.25$, $0.50$ and $0.75$ under the alternative.
For the number of basic components, $p$, they were chosen for each simulated sample such that the percentage of explained variance was at least 95\%. The most frequent values observed for $p$ were 9 and 10. This is close to the dimension with 95\% of variance in the Brownian bridge, which is the distribution of the covariate.

The empirical powers of the Kokoszka \emph{et al.} (2008)'s test are also shown titled $KMSZ$. The same dimension $p$ with 95\% of explained variance is used as the dimension of the covariate in their test. Their test also requires choosing a dimension of the response. The value of $2$ was taken in all cases.

From Table 2, one can conclude that the new test is generally respects the nominal levels, while Kokoszka \emph{et al.} (2008)'s test is somewhat conservative, specially for small samples.

Regarding the power under the alternative, one would expect Kokoszka \emph{et al.} (2008)'s test to be more powerful, since their test is specifically designed to check this type of effect. We consider that this power comparison is affected by the conservative nature of their test for a high dimension such as 9 or 10, as  estimated here. We will see this in more detail in the following experiments with several fixed values of the dimension $p$.

\medskip
\begin{center}
\emph{Insert Table 2 here}
\end{center}
\medskip


Table 3 shows the empirical powers of the new test and Kokoszka \emph{et al.} (2008)'s test for fixed values of the dimension $p$ under the null hypothesis. The option ``random'' for $p$ represents the random number of components required to obtain at least 95\% of the explained variance. The new test  respects the nominal levels for any dimension, while Kokoszka \emph{et al.} (2008)'s test is conservative for high dimensions, especially with small sample size and small nominal level.

\medskip
\begin{center}
\emph{Insert Table 3 here}
\end{center}
\medskip


Table 4 is similar to Table 3, but under the alternative hypothesis coming from the functional linear effect with $c=0.5$. As expected, Kokoszka \emph{et al.} (2008)'s test is more powerful for low dimensions, since this is the ideal situation of their parametric test. An increasing dimension produces a power loss in both tests, as a consequence of more noise in the statistic, while low dimensions are sufficient to detect the alternative. The new test is less affected by the dimension than the parametric test. In particular, the new test becomes more powerful than its parametric competitor for high dimensions. Although part of parametric test's power loss can be assigned to the asymptotic distribution inaccuracy, it is also true that the new test is designed to overcome the problem of dimension, usually called {\it curse of dimension} in the literature of lack-of-fit tests.

\medskip
\begin{center}
\emph{Insert Table 4 here}
\end{center}
\medskip


Other alternatives were considered to complete the comparison with Kokoszka \emph{et al.}'s test. One of them is of the following type:
$$U_i(t)=\beta(t) X_i(t)+\epsilon_i(t),\qquad 1\leq i\leq n,$$
where $X_i$ and $\epsilon_i$ are independent Brownian bridges (as in the previous situation) and $\beta$ is a square-integrable function on $[0,1]$. This is the so-called concurrent model studied in detail in Ramsay and Silverman (2005), where the covariate at time $t,$ $X_i(t),$ only influences the response function at time $t,$ $U_i(t)$. The function $\beta$ was $\beta(t)=c\cdot\exp(-4(t-0.3)^2)$, with $c=0$ under the null and $c=0.4$ under the alternative.

A completely nonlinear alternative was also considered. In this case a quadratic model of this type was generated:
$$U_i(t)=H\left(X_i(t)\right)+\epsilon_i(t),\qquad 1\leq i\leq n,$$
where $X_i$ and $\epsilon_i$ are independent Brownian motion and Brownian bridge, respectively, and $H(x)=c\cdot(x^2-1)$. The null hypothesis is satisfied when $c=0$, while the alternative is represented by $c=0.5$.

Table 5 contains the percentages of rejection under the three alternative models for both tests with different significance levels and sample sizes. The bandwidth followed the rule $h=n^{-2/9}$, and the dimension $p$ was taken to explain 95\% of the variance in the empirical PCA.

Kokoszka \emph{et al.}'s test is more powerful than the new test under the linear alternative, and also under the concurrent alternative. This is not necessarily surprising since the concurrent model is in a sense, a degenerate functional linear model. On the other hand, Kokozska \emph{et al.}'s test, which was designed to detect only linear effects, is not powerful under the quadratic alternative.

\medskip
\begin{center}
\emph{Insert Table 5 here}
\end{center}
\medskip


We shall now consider a functional linear model with heteroscedastic error, given by
$$U_i(t)=\int_0^1 \zeta(s,t) X_i(s)\,ds+\sqrt{1/2+X_i(t)^2}\epsilon_i(t),\qquad 1\leq i\leq n$$
where $X_i$ and $\epsilon_i$ are independent Brownian bridges and the kernel $\psi$ was chosen to be $\zeta(s,t)=c\cdot \exp(t^2+s^2)/2$, with $c=0$ under the null and $c=0.35$ under the alternative. It is the same functional linear model introduced in (\ref{eq:functional-linear}), but with heteroscedastic error. The new test and Kokoszka {\it et al.} (2008)'s test were applied with the same configuration used for the homoscedastic functional linear model.

Table 6 below contains the empirical powers under the null hypothesis for different values of the dimension $p$ and a random $p$ with 95\% of explained variance, with different significance levels and sample sizes. It can be seen that the new test  respects the nominal level, while Kokoszka {\it et al.} (2008)'s test provides significant  deviations from the nominal level. Note that their test assumes homoscedasticity.

\medskip
\begin{center}
\emph{Insert Table 6 here}
\end{center}
\medskip


Table 7 below shows the percentages of rejections obtained under the functional linear effect with heteroscedastic errors. Kokoszka \emph{et al.} (2008)'s test is more powerful than the new test for small dimensions $p$ since their test is designed for detecting linear deviations. Nevertheless, the new test is more robust for larger dimensions.

\medskip
\begin{center}
\emph{Insert Table 7 here}
\end{center}
\medskip


\subsubsection{Application to Canadian Weather data}

%

The methods proposed in this paper will be applied to check the goodness-of-fit of several models for the Canadian weather dataset. This dataset is included in the R package fda (http://www.r-project.org), which implements the methods for functional data analysis described in the book by Ramsay and Silverman (2005).

The data consist of the daily mean temperature and rainfall registered in 35 weather stations in Canada. A curve is available for each station, describing the rainfall for each day of the year. The same type of curve with the temperature is used as covariate to assist rainfall predictions. Several regression models with functional covariate and functional response have been studied in Ramsay and Silverman (2005), and illustrated using the Canadian weather dataset. The purpose here is to use our goodness-of-fit technique to assess the validity of each of the common models for this dataset. To this end, the residuals coming from the model to be tested are used in the expression of the test statistic, while the covariate is the temperature curve.

Table 8 below contains the $p-$values for testing the goodness-of-fit of the models. The stations are classified into four climatic zones: Atlantic, Pacific, Continental and Arctic, giving place to functional ANOVA models, together with the possible linear or concurrent effect of the functional covariate. The formulae in Table 8 describe the nature of each model, where $Y_{ij}(t)$ represents the logarithm of the rainfall at the station $i$ of the climate zone $j$ on day $t$, $X_{ij}(t)$ is the temperature at the same station on day $t$ of the year.

The models were estimated as described in Ramsay and Silverman (2005) using a Fourier basis, where eleven basis functions were taken to estimate the components $\alpha_j(t)$, and seven basis functions to estimate the concurrent or functional linear components.

For the application of the test, the bandwidth was chosen to be $h=1.0\,n^{-2/9}=0.4538$, and the Epanechnikov kernel was used. Estimated functional principal components of the covariate were used as a basis. The first two principal components were taken, since they were able to capture 96\% of the variability. The column entitled $\psi_1$ in Table 8 contains the $p-$values when the privileged direction is the first eigenvector, while the column entitled $\psi_u$ corresponds to an uninformative direction with the same coefficient in each basis function. The penalization was taken to be $\alpha_n=2$. The $p-$values are based on 999 bootstrap replicates.

The resulting $p-$values are all small, so some pattern different from the considered models should be explored. However, amongst the models included here, the really significant $p-$values correspond to the models that do not include the ANOVA effects, that is, the effects coming from the climatic zone. Then, using the new test, we can conclude that the climatic zone has a significant effect on the rainfall curve, and that this effect is mostly described by the considered parametric models.

\medskip
\begin{center}
\emph{Insert Table 8 here}
\end{center}
\medskip


\newpage

\section{Appendix: proofs}

\label{secproofs} \setcounter{equation}{0}
\color{black}In this section and in the Supplementary Material, \color{black} $c,c_1, C, C_1,...$  denote constants that may have different values from line to   line. Moreover, for any integrable function $\phi$ defined on the real line, $\mathcal{F}[\phi]$ denotes its Fourier Transform, that is $\mathcal{F}[\phi] (t) = \int_{\mathbb{R}} \phi(x) \exp\{-2\pi i t x\} dx.$
Finally, recall that if
$\gamma = (\gamma_1,\cdots,\gamma_p)\in \mathcal{S}^p,$ then
$
\langle X,\gamma\rangle  = \sum_{j=1}^p \langle X, \psi_j\rangle \gamma_j.
$

\bigskip

\begin{proof}[Proof of Lemma \ref{lem1}]
(A) We have
\begin{eqnarray*}
\mathbb{E} (U\mid X) = 0 & \Leftrightarrow & \mathbb{E} (\langle U , \psi_j \rangle \mid X) = 0, \; \forall j\geq 1\\
& \Leftrightarrow & \mathbb{E} (\langle U , \psi_j \rangle \mid \langle X, \psi_1\rangle,\cdots,\langle X, \psi_p\rangle) = 0, \; \forall j\geq 1, \forall p\geq 1\\
& \Leftrightarrow &
\mathbb{E} (\langle U , \psi_j \rangle \mid \langle X, \gamma \rangle ) = 0, \; \forall j\geq 1, \forall p\geq 1, \forall \gamma \in \mathcal{S}^p\\
& \Leftrightarrow & \mathbb{E} ( U  \mid \langle X, \gamma \rangle ) = 0, \; \forall p\geq 1, \forall \gamma \in \mathcal{S}^p\\
& \Leftrightarrow & \mathbb{E} ( U  \mid F_\gamma (\langle X, \gamma \rangle) ) = 0, \; \forall p\geq 1, \forall \gamma \in \mathcal{S}^p.
\end{eqnarray*}
The first and the fourth equivalence in the last display are due to the fact that $\mathcal{R}$ is a basis in $L^2[0,1]$. Next,  note that by Cauchy-Schwarz inequality $\forall j$,
$  \mathbb{E} |\langle U, \psi_{j}\rangle| \leq \mathbb{E} \| U \|  <\infty$.
Thus the second equivalence in the last display is
guaranteed elementary properties of the conditional expectations and the Doob's Martingale Convergence Theorem, while the third equivalence  is given by
Lemma 2.1-(A) of Lavergne and Patilea (2008). \color{black}
The justification of the equality $\mathbb{E} ( U  \mid \langle X, \gamma \rangle ) = \mathbb{E} ( U  \mid F_\gamma (\langle X, \gamma \rangle) )$ giving the last equivalence is provided in the Lemma A in the Supplementary Material.
\color{black} To complete the proof of part (A) it suffices to note that
\begin{eqnarray*}
\mathbb{E}\left[\langle U, \mathbb{E}\left(U \mid \langle X,\gamma \rangle \right) \rangle \right]&=&\esp \left[ \Vert \mathbb{E} ( U  \mid \langle X, \gamma \rangle ) \Vert^{2} \right]\\&=&\esp \left[ \Vert \mathbb{E} ( U  \mid  F_\gamma (\langle X, \gamma \rangle) ) \Vert^{2} \right]\\ &=&\mathbb{E}\left[\langle U, \mathbb{E}\left\{U \mid F_{\gamma}(\langle X,\gamma \rangle) \right\} \rangle \right].
\end{eqnarray*}

(B) First note that $\mathcal{A}_p \subset \bigcap_{j\geq 1}\mathcal{A}_{pj}$ where
$$
\mathcal{A}_{pj} = \{\gamma\in\mathcal{S}^p : \mathbb{E}(\langle U, \psi_j\rangle \mid \langle X, \gamma \rangle )=0 \,\, a.s.\, \}.
$$
Now, if $\mathbb{P} [\mathbb{E} (U\mid X) = 0] <1$, then  $\exists j\geq 1$ such that $\mathbb{P} [\mathbb{E} (\langle U, \psi_j\rangle \mid X) = 0] <1.$ Set $j\geq 1$ with this property and apply  Lemma 2.1 in Patilea, Saumard and Sanchez (2012) to  deduce that there exists $p_0\geq 1 $  such that, for any $p\geq p_0$, $\mathcal{A}_{pj}$  has Lebesgue measure zero on $\mathcal{S}^p$ and is not dense. Since $\mathcal{A}_p$ is included in any $\mathcal{A}_{pj}$, the conclusion follows.
\end{proof}

\bigskip

One of the ingredients we shall use for the proof of Lemma \ref{leem1} is an exponential bound for tail probabilities of $U-$statistics presented in Lemma \ref{gine_zinn} below and due to Gin\'e, Lata{\l}a and Zinn (2000). To state the result we shall use, let us  introduce some notation.
Let $Z_1,\cdots,Z_n$ be independent random variables (not necessarily with the same distribution) taking values in a measurable space $(\mathcal{Z}, \Upsilon )$.
Let $h_{i,j}(\cdot,\cdot)$, $1\leq i,j\leq n$ be real-valued measurable functions on $\mathcal{Z}^2$ such that $h_{i,j}(z_i,z_j) = h_{j,i}(z_j,z_i)$ and $\mathbb{E}[h_{i,j}(z_i,Z_j)]=0$,  $\forall 1\leq i,j\leq n$, $\forall z_i,z_j$. The functions $h_{i,j}$ could be different for different values of $n$. Define
\begin{equation}\label{norms_glz1}
A_n = \max_{i,j} \| h_{i,j}  (\cdot,\cdot) \|_\infty, \;\;\; B^2_n = \max_{j}\left\| \sum_{i} \mathbb{E}[h_{i,j}^2(Z_i,\cdot)]\right\|_\infty,  \;\;\; C_n^2 = \sum_{i,j} \mathbb{E}[h_{i,j}^2(Z_i,Z_j)],
\end{equation}
(herein $\|\cdot\|_\infty$ denote the sup norm) and
\begin{equation}\label{norms_glz2}
D_n = \sup\left\{ \mathbb{E}  \sum_{ i, j} h_{ i, j}(Z_i,Z_j)f_i(Z_i) g_j(Z_j) \; : \;
\mathbb{E}  \sum_{ i } f^2_i(Z_i)
  \leq 1, \; \mathbb{E}  \sum_{ j } g^2_j(Z_j)
   \leq 1 \right\}.
\end{equation}
The following result is a simplified version of Theorem 3.3 in Gin\'e, Lata{\l}a and Zinn (2000).

\begin{lem}\label{gine_zinn}
There exist a universal constant $L<\infty$ (in particular, independent on $n$ and the functions $h_{i,j}$) such that
$$
\mathbb{P} \left\{ \left|\sum_{1\leq i\neq j\leq n} h_{i,j}(Z_i,Z_j)  \right| \geq t \right\} \leq L \exp \left[ -\frac{1}{L} \min \left( \frac{t^2}{C_n^2}, \frac{t}{D_n}, \frac{t^{2/3}}{B_n^{2/3}}, \frac{t^{1/2}}{A_n^{1/2}}\right) \right], \qquad \forall t>0.
$$
\end{lem}

Let $\gamma\in\mathcal{S}^p$ and let $x_1,\cdots,x_n$ be an arbitrary collection of non-random points in $L^2[0,1].$
Consider $\widetilde Z_1,\cdots, \widetilde Z_n$ independent random variables with values in $L^2[0,1]$ such that for each $1\leq i\leq n $ the law of $\widetilde Z_i$
is the conditional law of $U_i$ given $X_i=x_i$.
We shall apply Lemma \ref{gine_zinn} with
$h_{i,i} \equiv 0$ and for $ 1\leq i\neq j\leq n$
\begin{equation}\label{hn_here}
h_{ i, j}(Z_i,Z_j) = \frac{\langle  Z_{i} \; ,  Z_{j}   \rangle}{n(n-1)hM^2}
K_{h}\left( F_{\gamma,n}(\langle x_{i},\gamma\rangle)- F_{\gamma,n}(\langle x_{j},\gamma\rangle) \right),
\end{equation}
where $ Z_{i} = \widetilde Z_i \mathbb{I}_{\{ \| \widetilde Z_i \| \leq M \}} - \mathbb{E}[\widetilde Z_i \mathbb{I}_{\{ \| \widetilde Z_i \| \leq M \}}]$, $M>0$ is some constant (that we shall allow to increase with $n$). (Note  in particular that the values $ \mathbb{E}[Z_i \mathbb{I}_{\{ \| Z_i \| \leq M \}}]$ coincide with the values $\mathbb{E}[U_i \mathbb{I}_{\{ \| U_i \| \leq M \}} \mid X_i = x_i].$)
Here $F_{\gamma,n}$ is the empirical d.f. of the sample $\langle x_1,\gamma \rangle,\cdots, \langle x_n,\gamma \rangle.$
The following lemma provides upper bounds for the quantities $A_n$ to $D_n$ in this setup. The bounds are independent of the collection
$x_1,\cdots,x_n\in L^2[0,1]$, and of $p\geq 1$ and $\gamma\in\mathcal{S}^p.$

The proof of the following lemma is given in the Supplementary Material.

\begin{lem}\label{upp_bds}
Under the conditions of Lemma \ref{leem1}, for $h_{i,j}$  defined as in (\ref{hn_here})
$$
A_n = \frac{\|K\|_\infty}{n (n-1)h}, \quad B^2_n \leq  \frac{c}{n^3h M^2}, \quad C^2_n \leq  \frac{c }{n^2hM^4} \quad \text{and} \quad D_n \leq  \frac{c}{n M^2},
$$
for some constant $c,$ only depending on the upper bound of $\mathbb{E}(\|U\|^2\mid X)$ and $\int K^2$.
\end{lem}

\bigskip

Another ingredient for the proof of Lemma \ref{leem1} is an upper bound for the number of different possible orderings in the sample $\langle X_1,\gamma \rangle, \cdots, \langle X_n,\gamma \rangle$ when $\gamma$ belongs to the unit hypersphere in $\mathbb{R}^p.$  Let $w_1,\cdots, w_n$ be a collection of $n$ points in $\mathbb{R}^p$ and let $\pi$ be a permutation of the set of integers $\{1,2,\cdots,n\}$. Following Cover (1967), we say that $\gamma \in \mathcal{S}^p$ induces the ordering $\pi$ if
$$
\langle w_{\pi(1)}, \gamma \rangle < \langle w_{\pi(2)}, \gamma \rangle < \cdots < \langle w_{\pi(n)}, \gamma \rangle.
$$
Conversely, the ordering $\pi$ will be said to be linearly inducible if such a vector $\gamma$ exists. The following result is due to Cover (1967).

\begin{lem}\label{cover67}
There are precisely $q(n,p)$ linearly inducible orderings of $n$ points in general position in $\mathbb{R}^p$, where
$$
q(n,p) =  2 \sum_{k=0}^{p-1} S_{n,k} = 2 \left[ 1 + \sum_{2\leq i \leq n-1} i + \sum_{2\leq i < j \leq n-1} ij + \cdots \right] \quad (\text{$p$ terms}),
$$
where $ S_{n,k}$ is the number of the $(n-2)! / (n-2 -k)! k!$ possible products of numbers taken $k$ at a time without repetition from the set $\{ 2,3,\cdots,n-1 \}$
\end{lem}

By Lemma \ref{cover67} we obtain a simple upper bound for $q(n,p)$ when $n\geq 2p$, \emph{i.e.,}
\begin{equation}
q(n,p) \leq 2[ 1+n^2+\cdots + n^{p}] \leq n^{p+1}.
\end{equation}

\bigskip

\begin{proof}[Proof of Lemma \ref{leem1}]
Set $M$ that depends on $n$ in a way that will be specified below.
Let
\begin{equation*}
Q_{M,n}\left(\gamma \right) =\frac{1}{n(n-1)}\sum\limits_{1\leq i\neq
j\leq n}\langle U_{M,i} , U_{M,j} \rangle \frac{1}{h}
K_{h}\left( F_{\gamma,n}(\langle X_{i},\gamma\rangle)- F_{\gamma,n}(\langle X_{j},\gamma\rangle) \right),\quad \gamma\in\mathcal{S}^p,
\end{equation*}
where $ U_{M,i}=U_i \mathbb{I}_{\{ \| U_i \| \leq M \}} - \mathbb{E}[ U_i \mathbb{I}_{\{ \| U_i \| \leq M \}}\mid X_i].$
We can write
$$
\mathbb{P}\left( \sup_{\gamma \in \mathcal{S}^p}|Q_{M,n}(\gamma)| > \frac{t  p\ln n}{nh^{1/2}} \right) = \mathbb{E} \left[\mathbb{P}\left( \sup_{\gamma \in \mathcal{S}^p}|Q_{M,n}(\gamma)| >\frac{t  p\ln n}{nh^{1/2}}\mid  X_1,\cdots, X_n\right)\right]
$$
In view of Lemma \ref{cover67}, for any $n,p$, given $X_1,\cdots, X_n$  there exists a set $\mathcal{O}_{np}\subset \mathbb{R}^p$ with at most $n^p$ elements, that depend on $X_1,\cdots, X_n$, such that
$$
\sup_{\gamma \in \mathcal{S}^p}|Q_{M,n}(\gamma)| = \sup_{\gamma \in \mathcal{O}_{np}}|Q_{M,n}(\gamma)|.
$$
Let $b_n = M^{-2} n^{-1}h^{-1/2}p\ln n$. By Lemmas \ref{gine_zinn} and \ref{cover67}, and Boole's inequality, deduce that there exists an universal constant $L$ such that for any $t>0$,
\begin{multline*}
\mathbb{P}\!\left( \sup_{\gamma \in \mathcal{S}^p}|Q_{M,n}(\gamma)| >\frac{t  p\ln n}{nh^{1/2}} \mid  X_1,\cdots, X_n \right)
 \leq \! \sum_{\gamma \in \mathcal{O}_{np}}
\!\mathbb{P}\!\left( |M^{-2}Q_{M,n}(\gamma)| > t  b_n \mid  X_1,\cdots, X_n\right)\\
\leq \max\{L,1\} \exp \left[ (p+1)\ln n -\frac{1}{L}  \min \left( \frac{(tb_n)^2}{C_n^2}, \frac{tb_n}{D_n}, \frac{(tb_n)^{2/3}}{B_n^{2/3}}, \frac{(tb_n)^{1/2}}{A_n^{1/2}}\right) \right].
\end{multline*}
Now, take $M = n^{1/4 - a}$ for some (small) $a>0$ and note that the exponential bound in the last display is independent of $X_1,\cdots, X_n$ and tends to zero for any $t$. Deduce that
$$
\sup_{\gamma \in \mathcal{S}^p}|Q_{M,n}(\gamma)| = O_{\mathbb{P}} (n^{-1}h^{-1/2}p\ln n).
$$
We next show that $\sup_{\gamma \in \mathcal{S}^p}|Q_{n}(\gamma) -Q_{M,n}(\gamma)| = o_{\mathbb{P}} (n^{-1}h^{-1/2}p\ln n)$. Let
\begin{equation*}
R_{1n}(\gamma) = \frac{1}{n(n-1)}\sum\limits_{1\leq i\neq
j\leq n}\!\!\langle U_{M,i} , U_j -  U_{M,j} \rangle \frac{1}{h}
K_{h}\left( F_{\gamma,n}(\langle X_{i},\gamma\rangle)\!-\! F_{\gamma,n}(\langle X_{j},\gamma\rangle) \right),\quad \!\gamma\in\mathcal{S}^p,
\end{equation*}
and
$R_{2n}(\gamma) = Q_{n}(\gamma) - Q_{M,n}(\gamma) - 2R_{1n}(\gamma). $
Since $K(\cdot)$ is bounded, we have
\begin{multline*}
\mathbb{E} \left[\sup_{\gamma }\left\vert R_{1n}(\gamma)\right\vert\mid X_1,\cdots,X_n\right]
\leq Ch^{-1}\mathbb{E}\left( \| U_{M,i} \| \| U_j -U_{M,j}
\| \right) \\
\leq 2Ch^{-1} \mathbb{E}\left( \| U
_{i}\| \right) \mathbb{E}\left( \| U_j -U_{M,j}\|
\right) .
\end{multline*}
By the H\"{o}lder inequality and the Chebyshev inequality
\begin{equation*}
\mathbb{E}\left(  \| U_j -U_{M,j}\| \right)  \leq 2\mathbb{E}^{1/m}\left[
\| U _{j}\| ^{m}\right] \mathbb{P}^{(m-1)/m}\left[
\| U _{j}\| >M\right] \leq 2\mathbb{E}\left[
\| U _{j}\| ^{m}\right] \,M^{1-m}.
\end{equation*}
Now, to deduce that $R_{1n}(\gamma)$ is uniformly negligible, it suffices to note that under Assumption \ref{K}-(b), for $m>7$ and $a$ sufficiently small,  $$M^{1-m}= n^{(1-m)(1/4 - a)} = o\left( n^{-1}h^{1/2}p\ln n\right) .$$ Clearly,  $\sup_{\gamma }|R_{2n}(\gamma)|$ is of smaller order than $\sup_{\gamma }|R_{1n}(\gamma)|$.

For the inverse of the variance estimator, for any $\gamma\in\mathcal{S}^p,$ let us define
\begin{equation*}
\widehat{v}_{N,n}^{2}( \gamma) =\frac{2}{n(n-1)h}\sum\limits_{j\neq
i}\langle U_{i},U_{j}\rangle^2 \mathbb{I}_{\{ \langle U_{i},U_{j}\rangle^2\leq N \}}
K_{h}^{2}\left( F_{\gamma,n}(\langle X_{i},\gamma\rangle)- F_{\gamma,n}(\langle X_{j},\gamma\rangle) \right).
\end{equation*}
Using the H\"{o}lder inequality,  the Chebyshev inequality and the Cauchy-Schwarz inequality,
\begin{eqnarray*}
\mathbb{E} \left[\sup_{\gamma }\left\vert \widehat{v}_{n}^{2}( \gamma)  - \widehat{v}_{N,n}^{2}( \gamma) \right\vert  \mid X_1,\cdots,X_n\right]
&\leq &\!\! Ch^{-1}\mathbb{E}\left(  \langle U_{i},U_{j}\rangle^2 \mathbb{I}_{\{ \langle U_{i},U_{j}\rangle^2 > N \}}
 \right) \\
 &\leq & \!\! h^{-1}\mathbb{E}^{1/s}\left[
\langle U_{i},U_{j}\rangle^{2s}\right] \mathbb{P}^{(s-1)/s}\left[
\langle U_{i},U_{j}\rangle^{2s} >N^s \right] \\
&\leq & \!\! h^{-1}\mathbb{E}^2\left[
\| U _{j}\| ^{2s}\right] \,N^{1-s}.
\end{eqnarray*}
Take $s=4,$ $N= n^{1/4}$ and deduce that the right bound in the last display tends to zero.
On the other hand, we apply the Hoeffding (1963) inequality for $U-$statistics to control the deviations of $\widehat{v}_{N,n}^{2}( \gamma) - \mathbb{E}[ \widehat{v}_{N,n}^{2}( \gamma) \mid X_1,\cdots,X_n]$ conditionally on $X_1,\cdots,X_n.$ For any fixed $\gamma$ we have
\begin{multline*}
\mathbb{P}\left( n^{1/2}h |\widehat{v}_{N,n}^{2}( \gamma) - \mathbb{E}[ \widehat{v}_{N,n}^{2}( \gamma) \mid  X_1,\cdots,X_n]|  \geq t \mid  X_1,\cdots,X_n \right)\\ \leq 2 \exp\left\{ - \frac{[n/2]n^{-1} t^2 }{2[\tau^2 +  K^2(0) N n^{-1/2} t/3]} \right\}
\end{multline*}
where $\tau^2$ is the variance of a term in the sum defining $h\widehat{v}_{N,n}^{2}( \gamma) - \mathbb{E}[ h\widehat{v}_{N,n}^{2}( \gamma) \mid X_1,\cdots,X_n]$.
Take $t=n^{1/2 -c}h$ for some small $c>0$ and note that $\tau^2 \leq C $ for some constant independent of  $\gamma$ and $h$.
In the similar way as we did for $Q_{M,n}(\gamma)$, applying Lemma \ref{cover67}, we obtain an exponential bound for the tail of
$\widehat{v}_{N,n}^{2}( \gamma) - \mathbb{E}[ \widehat{v}_{N,n}^{2}( \gamma) \mid X_1,\cdots,X_n]$ given $X_1,\cdots,X_n$
\emph{uniformly} with respect to $\gamma.$ This bound is independent of  $X_1,\cdots,X_n.$  Deduce that
$$
\sup_{\gamma }|\widehat{v}_{N,n}^{2}( \gamma) - \mathbb{E}[ \widehat{v}_{N,n}^{2}( \gamma) \mid  X_1,\cdots,X_n]|  = o_{\mathbb{P}} (1),
$$
conditionally on  $X_1,\cdots,X_n$ and unconditionally.
It remains to  note that Assumption \ref{D}-(c) and the first part of Lemma \ref{db_integr} in the Supplementary Material guarantee that $\mathbb{E}[ \widehat{v}_{N,n}^{2}( \gamma) \mid  X_1,\cdots,X_n]$ stays away from zero. Taking all the results together, conclude that $1/\widehat{v}_{n}^{2}( \gamma)$ is uniformly bounded in probability.
\end{proof}

\bigskip

The proof of Lemma \ref{beta} is similar to the proof of Lemma 3.2 in Lavergne and Patilea (2008) and hence will be omitted.
The proofs of Theorem \ref{as_law}, \ref{altern}, \ref{bbot} and Corollary \ref{FPC} are given in the Supplementary Material.

\newpage

\begin{center}
{\small REFERENCES }
\end{center}

{\small \parindent=0cm }

\newcounter{refs}
\begin{list}{}{\usecounter{refs}
\setlength{\labelwidth}{0in}
\setlength{\labelsep}{0in}\setlength{\leftmargin}{0in}
\setlength{\rightmargin}{0in} \setlength{\topsep}{0in}
\setlength{\partopsep}{0in} \setlength{\itemsep}{0cm} }

\item {\footnotesize \textsc{Antoch, J., Prchal, L., De Rosa, M., and Sarda, P.} (2008) Functional linear regression with functional response: application to prediction of electricity consumption.  \textsl{International Workshop on Functional and Operatorial Statistics 2008 Proceedings, Functional and operatorial statistics, Dabo-Niang and Ferraty (Eds.), Physica-Verlag, Springer.}}


\item {\footnotesize \textsc{Bosq, D.} (2000). \textsl{Linear Processes in Function Spaces: Theory and Applications}. Lecture Notes in Statistics (v. 149). Springer-Verlag, New-York.}

\item {\footnotesize \textsc{Chiou, J-M., and M\"{u}ller, H-G.} (2007).
Diagnostics for Functional Regression
via Residual Processes. \textsl{Comput. Statist. Data Anal.} \textbf{15},
4849--4863.}

\item {\footnotesize \textsc{Chiou, J.H., M\"{u}ller, H.G., Wang, J.L., and Carey, J.R.} (2003). A functional
multiplicative effects model for longitudinal data, with applications to
reproductive histories of female medflies. \textsl{Statist. Sinica} \textbf{13},  1119--1133.}



\item {\footnotesize \textsc{Cover, T.M.} (1967). The number of linearly inducible orderings in $d-$space. \textsl{SIAM J. Appl. Math.} \textbf{15}, 434--439.}

\item {\footnotesize \textsc{Crambes, C., and Mas, A.} (2009). Asymptotics of prediction in the functional linear regression with functional outputs.
arXiv:0910.3070v3 [math.ST]}

\item {\footnotesize \textsc{Cuesta-Albertos, J.A., del Barrio, E., Fraiman, R., and Matr\'{a}n, C.}
(2007). The random projection method in goodness of fit for functional data.
\textsl{Comput. Statist. Data Anal.} \textbf{51},
4814--4831.}

\item {\footnotesize \textsc{Cuesta-Albertos, E., Fraiman, R., and Ransford, T.}
(2007). A sharp form of the Cram\'er-Wold theorem.
\textsl{J. Theoret. Probab.} \textbf{20},
201--209.}

\item {\footnotesize \textsc{Cuevas, A., Febrero, M., and Fraiman, R.} (2002). Linear functional regression: The case of fixed design and functional response. \textsl{Canadian J. Statist.} \textbf{30}, 285--300. }

\item {\footnotesize \textsc{Delsol, L., Ferraty, F., and Vieu, P.} (2011).
Structural test in regression on functional variables. \textsl{J. Multivariate Anal.} \textbf{102}, 422--447.}

\item {\footnotesize \textsc{de Jong, P.} (1987). A central limit theorem for generalized quadratic forms. \textsl{Probab. Theory Related Fields} \textbf{75}, 261--277.}

\item {\footnotesize \textsc{Fan, Y., and Li, Q.} (1996). Consistent model
specification tests: omitted variables and semiparametric functional forms.
\textsl{Econometrica} \textbf{64}, 865--890.}

\item {\footnotesize \textsc{Ferraty, F., Laksaci, A., Tadj, A., and Vieu, P.} (2011).
Kernel regression with functional response. \textsl{Electron. J. Stat.} \textbf{5}, 159--171.}

\item {\footnotesize \textsc{Ferraty, F., Van Keilegom, I., and Vieu, P.} (2012). Regression when both response and predictor are functions. \textsl{J. Multivariate Anal.} \textbf{109}, 10--28.}

\item {\footnotesize \textsc{Ferraty, F., and Vieu, P.} (2006).
     \textsl{Nonparametric Functional Data Analysis: Theory and Practice}.
     Springer,  Berlin.}

\item {\footnotesize \textsc{Gabrys, R., Horv\'ath, L.,  and Kokoszka, P.} (2010). Tests for error correlation in the functional linear model. \textsl{J. Amer. Statist.  Assoc.} \textbf{105}, 1113--1125.}

\item {\footnotesize \textsc{Gao, J., and Gijbels, I.} (2008). Bandwidth selection in nonparametric kernel testing. \textsl{J. Amer. Statist.  Assoc.} \textbf{103}, 1584--1594.}

\item {\footnotesize \textsc{Garc\'{i}a-Portugu\'{e}s, E., Gonz\'{a}lez-Manteiga, W., and Febrero-Bande, M.} (2012)
A goodness-of-fit test for the functional linear model with scalar response. arXiv:1205.6167v3 [stat.ME]}

\item {\footnotesize \textsc{Gin\'{e}, E., Lata{\l}a, R., and Zinn, J.} (2000). Exponential and moment inequalities for $U-$statistics. In \textsl{High Dimensional II} 13--38. Progr. Probab. \textbf{47}. Birkh\"{a}user, Boston. }


\item {\footnotesize \textsc{H\"{a}rdle, W., and Mammen, E.} (1993).
Comparing nonparametric versus parametric regression fits. \textsl{Ann.
Statist.} \textbf{21}, 1296-1947. }

\item {\footnotesize \textsc{Hoeffding, W.} (1963). Probability inequalities for sums of bounded random variables. \textsl{J. Amer. Statist.  Assoc.} \textbf{58}, 13--30.}

\item {\footnotesize \textsc{Horowitz, J., and Spokoiny, V.} (2001). An adaptive rate-optimal test of a parametric mean-regression model against a nonparametric alternative.  \textsl{Econometrica} \textbf{69}, 599--633.
}

 \item {\footnotesize \textsc{Horv\'{a}th, L., and Kokoszka, P.} (2012). \textsl{Inference for Functional Data with Applications.} Springer Series in Statistics. Springer, New-York. }

\item {\footnotesize \textsc{Kokoszka, P., Maslova, I., Sojka, J., and Zhu, L.} (2008). Testing for lack of dependence in the functional linear model. \textsl{Canadian J. Statist.} \textbf{36}, 1--16. }

\item {\footnotesize \textsc{Lavergne, P., and Patilea, V.} (2008). Breaking the curse of dimensionality in nonparametric testing. \textsl{J. Econometrics} \textbf{143}, 103--122. }

\item {\footnotesize \textsc{Lian, H.} (2011).
Convergence of functional $k-$nearest neighbor regression estimate with functional responses. \textsl{Electron. J. Statist.} \textbf{5}, 31--40.}

\item {\footnotesize \textsc{Mammen, E.} (1993). Bootstrap and Wild Bootstrap for High Dimensional Linear Models. \textsl{Ann. Statist.} \textbf{21}, 255--285.}


\item {\footnotesize \textsc{Patilea, V., S\'{a}nchez-Sellero, C., and Saumard, M.} (2012). Projection-based nonparametric goodness-of-fit testing with functional covariates.
arXiv:1205.5578  [math.ST]}

\item {\footnotesize \textsc{Ramsay, J., and Silverman, B.W.} (2005).
     \textsl{Functional Data Analysis} (2nd ed.).
     Sprin\-ger-Verlag, New York.}


\item {\footnotesize \textsc{Stute, W.} (1984). Asymptotic normality of nearest neighbor regression function estimates. \textsl{Ann. Statist.} \textbf{12}, 917--926.}

\item {\footnotesize \textsc{Stute, W., Gonz\'alez-Manteiga, W.} (1996). NN goodness-of-fit tests for linear models.  \textsl{J. Statist. Plann. Inference} \textbf{53}, 75--92.}


\item{\footnotesize \textsc{Yang, S.S.} (1981). Linear Functions of Concomitants of Order Statistics with Application to Nonparametric Estimation of a Regression Function. \textsl{J. Amer. Statist.  Assoc.} \textbf{76}, 658--662.}

\item {\footnotesize \textsc{Yao, F.,  M\"{u}ller, H.G., and Wang, J-L.} (2005). Functional linear regression analysis for longitudinal data. \textsl{Ann. Statist.} \textbf{33}, 2873--2903.}

\item {\footnotesize \textsc{Zheng, J.X.} (1996). A consistent test of
functional form via nonparametric estimation techniques. \textsl{J.
Econometrics} \textbf{75}, 263--289. }

\end{list}

\newpage

$$\begin{tabular}{cccccccccc}
&&\multicolumn{2}{c}{Level\ $=10\%$} && \multicolumn{2}{c}{Level\ $=5\%$}
&& \multicolumn{2}{c}{Level\ $=1\%$} \\
Hypothesis & $c_h$ & $n=100$ & $n=200$ && $n=100$ & $n=200$ && $n=100$ & $n=200$ \\
\hline  &&&&&&&&& \\ 
    Null    & 0.75 &  11.7 &  9.9 &&  6.0 &  5.1 &&  1.5 &  0.8 \\
            &  1.0 &  12.2 & 10.2 &&  5.6 &  5.0 &&  1.1 &  1.0 \\
         \vspace*{1.5mm}
            & 1.25 &  12.2 & 10.1 &&  5.7 &  5.0 &&  1.1 &  1.2 \\
Alternative & 0.75 &  66.3 & 90.9 && 53.7 & 84.1 && 27.6 & 65.6 \\
            &  1.0 &  70.4 & 91.5 && 57.6 & 85.6 && 29.4 & 69.4 \\
         \vspace*{1.5mm}
            & 1.25 &  73.1 & 93.1 && 59.9 & 87.1 && 31.4 & 72.7 \\
 \hline
\end{tabular}$$
\begin{center}
{\bf Table 1.} Percentages of rejections for the new test under the univariate predictor model.
\end{center}

\qquad

\qquad

\qquad

$$\begin{tabular}{cccccccccc}
&&\multicolumn{2}{c}{Level\ $=10\%$} && \multicolumn{2}{c}{Level\ $=5\%$}
&& \multicolumn{2}{c}{Level\ $=1\%$} \\
$c$ & Test & $n=40$ & $n=100$ && $n=40$ & $n=100$ && $n=40$ & $n=100$ \\
\hline &&&&&&&&& \\ 
 0    &  New & 10.5 & 11.1 &&  4.9 &  5.6 &&  1.2 & 0.8 \\ \vspace*{1.5mm}
      & KMSZ &  7.1 &  8.0 &&  3.2 &  3.5 &&  0.0 & 0.7 \\
 0.25 &  New & 18.5 & 31.7 && 10.4 & 21.7 &&  3.4 & 8.7 \\ \vspace*{1.5mm}
      & KMSZ & 14.5 & 27.1 &&  5.9 & 17.2 &&  0.5 & 3.7 \\
 0.50 &  New & 44.5 & 85.2 && 33.7 & 78.1 && 14.1 & 59.4 \\ \vspace*{1.5mm}
      & KMSZ & 45.7 & 85.3 && 27.4 & 75.0 &&  6.9 & 51.1 \\
 0.75 &  New & 79.3 & 99.7 && 70.1 & 99.4 && 45.5 & 97.8 \\ \vspace*{1.5mm}
      & KMSZ & 83.0 & 99.9 && 70.5 & 99.7 && 36.7 & 98.1 \\
 \hline
\end{tabular}$$
\begin{center}
{\bf Table 2.} Percentages of rejections for the new test and Kokoszka \emph{et al.} (2008)'s test under the null hypothesis ($c=0$) and different deviations (represented by $c$) under the alternative.
\end{center}
\vspace*{3mm}

\newpage

$$\begin{tabular}{cccccccccc}
&&\multicolumn{2}{c}{Level\ $=10\%$} && \multicolumn{2}{c}{Level\ $=5\%$}
&& \multicolumn{2}{c}{Level\ $=1\%$} \\
$p$ & Test & $n=40$ & $n=100$ && $n=40$ & $n=100$ && $n=40$ & $n=100$ \\
\hline &&&&&&&&& \\
 2 &  New &  9.7 & 10.6 && 5.1 & 5.2 && 0.5 & 1.1 \\ \vspace*{1.5mm}
   & KMSZ &  9.1 &  9.7 && 3.9 & 5.3 && 0.9 & 0.4 \\
 3 &  New &  9.6 & 10.7 && 5.2 & 5.7 && 1.0 & 1.3 \\ \vspace*{1.5mm}
   & KMSZ &  9.6 & 10.7 && 4.0 & 4.9 && 0.6 & 0.4 \\
 5 &  New & 10.1 & 10.4 && 5.2 & 5.9 && 0.9 & 1.1 \\ \vspace*{1.5mm}
   & KMSZ &  8.1 & 10.0 && 3.9 & 4.9 && 0.2 & 0.6 \\
 7 &  New & 10.9 & 11.2 && 5.5 & 5.1 && 1.2 & 0.9 \\ \vspace*{1.5mm}
   & KMSZ &  7.9 &  9.3 && 3.3 & 3.8 && 0.3 & 0.4 \\
10 &  New & 10.9 & 10.8 && 5.6 & 5.5 && 1.4 & 0.9 \\ \vspace*{1.5mm}
   & KMSZ &  7.4 &  8.3 && 3.1 & 3.7 && 0.0 & 0.7 \\
15 &  New & 11.2 & 10.8 && 7.5 & 5.8 && 1.9 & 1.3 \\ \vspace*{1.5mm}
   & KMSZ &  5.5 &  7.6 && 1.3 & 3.6 && 0.0 & 0.3 \\
Random &  New & 10.5 & 11.1 && 4.9 & 5.6 && 1.2 & 0.8 \\
       & KMSZ &  7.1 &  8.0 && 3.2 & 3.5 && 0.0 & 0.7 \\
 \hline
\end{tabular}$$
\begin{center}
{\bf Table 3.} Percentages of rejections for the new test and Kokoszka \emph{et al.} (2008)'s test under the null hypothesis.
\end{center}

\qquad

\qquad

\qquad

$$\begin{tabular}{cccccccccc}
&&\multicolumn{2}{c}{Level\ $=10\%$} && \multicolumn{2}{c}{Level\ $=5\%$}
&& \multicolumn{2}{c}{Level\ $=1\%$} \\
$p$ & Test & $n=40$ & $n=100$ && $n=40$ & $n=100$ && $n=40$ & $n=100$ \\
\hline &&&&&&&&& \\ 
 2 &  New & 62.7 & 97.5 && 45.8 & 94.1 && 21.6 & 80.9 \\ \vspace*{1.5mm}
   & KMSZ & 75.8 & 98.7 && 64.8 & 98.1 && 39.7 & 90.8 \\
 3 &  New & 53.3 & 94.8 && 42.1 & 89.7 && 20.1 & 75.9 \\ \vspace*{1.5mm}
   & KMSZ & 69.6 & 98.2 && 57.2 & 96.5 && 30.2 & 85.4 \\
 5 &  New & 50.4 & 92.2 && 38.5 & 87.6 && 16.5 & 71.3 \\ \vspace*{1.5mm}
   & KMSZ & 59.5 & 94.9 && 45.6 & 91.9 && 17.1 & 74.5 \\
 7 &  New & 47.4 & 91.1 && 35.2 & 85.9 && 15.7 & 67.6 \\ \vspace*{1.5mm}
   & KMSZ & 52.3 & 92.2 && 35.6 & 85.7 &&  9.8 & 66.2 \\
10 &  New & 43.1 & 86.1 && 31.5 & 79.5 && 13.4 & 60.0 \\ \vspace*{1.5mm}
   & KMSZ & 39.7 & 86.4 && 23.2 & 77.3 &&  3.9 & 52.4 \\
15 &  New & 38.0 & 80.7 && 27.3 & 72.0 && 11.3 & 48.8 \\ \vspace*{1.5mm}
   & KMSZ & 25.6 & 75.9 && 10.9 & 65.1 &&  0.4 & 36.6 \\
Random &  New & 44.5 & 85.2 && 33.7 & 78.1 && 14.1 & 59.4 \\
       & KMSZ & 45.7 & 85.3 && 27.4 & 75.0 &&  6.9 & 51.1 \\
 \hline
\end{tabular}$$
\begin{center}
{\bf Table 4.} Percentages of rejections for the new test and Kokoszka \emph{et al.} (2008)'s test under the functional linear effect.
\end{center}

\newpage

$$\begin{tabular}{cccccccccc}
&&\multicolumn{2}{c}{Level\ $=10\%$} && \multicolumn{2}{c}{Level\ $=5\%$}
&& \multicolumn{2}{c}{Level\ $=1\%$} \\
Model & Test & $n=40$ & $n=100$ && $n=40$ & $n=100$ && $n=40$ & $n=100$ \\
\hline &&&&&&&&& \\
Linear     &  New & 44.5 & 85.2 && 33.7 & 78.1 && 14.1 & 59.4 \\ \vspace*{1.5mm}
           & KMSZ & 45.7 & 85.3 && 27.4 & 75.0 &&  6.9 & 51.1 \\
Concurrent &  New & 28.5 & 55.3 && 19.9 & 43.5 &&  7.4 & 22.8 \\ \vspace*{1.5mm}
           & KMSZ & 34.8 & 80.3 && 20.9 & 68.8 &&  4.5 & 38.7 \\
Quadratic  &  New & 45.7 & 99.0 && 31.1 & 97.7 && 10.7 & 92.4 \\ \vspace*{1.5mm}
           & KMSZ & 28.6 & 29.5 && 17.3 & 20.8 &&  4.8 &  9.6 \\
 \hline
\end{tabular}$$
\begin{center}
{\bf Table 5.} Percentages of rejections for the new test and Kokoszka \emph{et al.} (2008)'s test under three alternative models.
\end{center}

\qquad

\qquad

\qquad

$$\begin{tabular}{cccccccccc}
&&\multicolumn{2}{c}{Level\ $=10\%$} && \multicolumn{2}{c}{Level\ $=5\%$}
&& \multicolumn{2}{c}{Level\ $=1\%$} \\
$p$ & Test & $n=40$ & $n=100$ && $n=40$ & $n=100$ && $n=40$ & $n=100$ \\
\hline &&&&&&&&& \\
 2 &  New &  8.8 & 10.8 && 4.5 &  5.4 && 0.5 & 1.3 \\ \vspace*{1.5mm}
   & KMSZ & 16.4 & 18.1 && 8.8 & 10.5 && 1.8 & 2.9 \\
 3 &  New &  9.6 &  9.9 && 5.2 &  5.4 && 0.7 & 1.1 \\ \vspace*{1.5mm}
   & KMSZ & 15.6 & 16.0 && 8.1 &  9.9 && 1.2 & 2.7 \\
 5 &  New &  9.7 & 10.4 && 4.8 &  4.8 && 0.7 & 1.0 \\ \vspace*{1.5mm}
   & KMSZ & 12.7 & 14.9 && 5.9 &  8.1 && 0.6 & 1.5 \\
 7 &  New &  9.8 &  9.3 && 4.6 &  4.3 && 0.7 & 1.2 \\ \vspace*{1.5mm}
   & KMSZ & 11.0 & 13.3 && 4.8 &  6.2 && 0.7 & 1.2 \\
10 &  New & 11.0 & 10.2 && 5.5 &  5.0 && 1.5 & 0.9 \\
   & KMSZ &  9.7 & 11.0 && 3.9 &  5.7 && 0.3 & 1.1 \\ \vspace*{1.5mm}
15 &  New & 11.8 & 10.6 &&  7.1 &  5.4 &&  1.9 &  1.6 \\ \vspace*{1.5mm}
   & KMSZ &  7.5 &  9.5 &&  2.0 &  4.9 &&  0.0 &  0.6 \\
Random &  New &  9.9 & 10.2 &&  5.1 &  5.3 &&  1.4 &  0.9 \\
       & KMSZ & 10.2 & 11.3 &&  4.2 &  5.8 &&  0.3 &  1.1 \\
   \hline
\end{tabular}$$
\begin{center}
{\bf Table 6.} Percentages of rejections for the new test and Kokoszka \emph{et al.} (2008)'s test under the null hypothesis with heteroscedastic error.
\end{center}

\qquad

\qquad

\qquad

$$\begin{tabular}{cccccccccc}
&&\multicolumn{2}{c}{Level\ $=10\%$} && \multicolumn{2}{c}{Level\ $=5\%$}
&& \multicolumn{2}{c}{Level\ $=1\%$} \\
$p$ & Test & $n=40$ & $n=100$ && $n=40$ & $n=100$ && $n=40$ & $n=100$ \\
\hline &&&&&&&&& \\
 2 &  New & 43.9 & 88.4 && 29.7 & 78.9 && 12.5 & 55.0 \\ \vspace*{1.5mm}
   & KMSZ & 63.6 & 91.6 && 51.7 & 86.7 && 29.5 & 69.4 \\
 3 &  New & 38.7 & 81.0 && 27.9 & 72.3 &&  9.9 & 49.5 \\ \vspace*{1.5mm}
   & KMSZ & 58.9 & 88.3 && 45.6 & 82.1 && 21.8 & 63.3 \\
 5 &  New & 34.8 & 77.0 && 25.1 & 68.5 &&  9.3 & 44.4 \\ \vspace*{1.5mm}
   & KMSZ & 50.0 & 81.7 && 35.6 & 73.5 && 12.0 & 52.7 \\
 7 &  New & 34.8 & 75.0 && 22.9 & 65.8 &&  7.8 & 39.2 \\ \vspace*{1.5mm}
   & KMSZ & 42.5 & 77.0 && 27.1 & 66.6 &&  7.0 & 42.1 \\
10 &  New & 31.6 & 69.3 && 21.0 & 58.6 &&  7.1 & 34.2 \\ \vspace*{1.5mm}
   & KMSZ & 33.0 & 69.5 && 18.6 & 56.5 &&  2.5 & 31.9 \\
15 &  New & 27.2 & 61.6 && 18.8 & 48.6 &&  5.5 & 27.5  \\ \vspace*{1.5mm}
   & KMSZ & 21.5 & 59.1 &&  7.5 & 44.0 &&  0.3 & 19.7  \\
Random &  New & 31.9 & 68.5 && 22.4 & 57.8 &&  7.5 & 32.9 \\
       & KMSZ & 37.0 & 67.5 && 21.9 & 54.5 &&  4.4 & 30.3 \\
 \hline
\end{tabular}$$
\begin{center}
{\bf Table 7.} Percentages of rejections for the new test and Kokoszka \emph{et al.} (2008)'s test under the functional linear effect with heteroscedastic error.
\end{center}

\qquad

\qquad

\qquad

$$\begin{tabular}{lccc}
&& \multicolumn{2}{c}{$p-$values} \\
Name of the model & Formula & $\psi_1$ & $\psi_u$ \\ \hline
No-effect & $Y_{ij}(t)=\mu(t)+\varepsilon_{ij}(t)$ & 0.0 & 0.0 \\
Functional ANOVA & $Y_{ij}(t)=\mu(t)+\alpha_j(t)+\varepsilon_{ij}(t)$ & 10.5 & 8.6 \\
Concurrent model &
$Y_{ij}(t)=\mu(t)+X_{ij}(t)\beta(t)+\varepsilon_{ij}(t)$ & 0.0 & 0.0 \\
Concurrent ANCOVA &
$Y_{ij}(t)=\mu(t)+\alpha_j(t)+X_{ij}(t)\beta(t)+\varepsilon_{ij}(t)$ & 9.6 & 6.6 \\
Functional linear model &
$Y_{ij}(t)=\mu(t)+\int X_{ij}(s)\beta(s,t)\,ds +\varepsilon_{ij}(t)$ & 0.0 & 0.0 \\
Functional ANCOVA &
$Y_{ij}(t)=\mu(t)+\alpha_j(t)+\int X_{ij}(s)\beta(s,t)\,ds +\varepsilon_{ij}(t)$ & 10.9 & 8.7 \\
 \hline
\end{tabular}$$
{\bf Table 8.} $p-$values (in percentages) for each of the models.

\newpage

\begin{figure}[h!]
\begin{center}
\includegraphics[width=8cm]{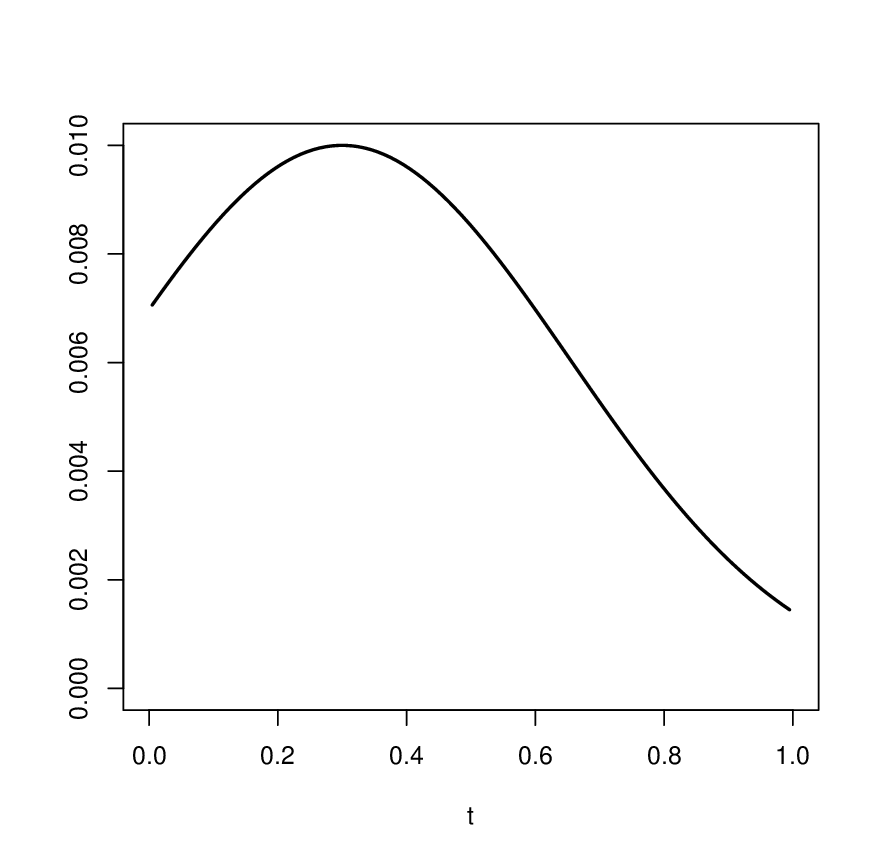}
\caption{Function $\mu(t)=0.01\cdot\exp(-4\cdot(t-0.3)^2)$ representing the common curve shape under the multiplicative effects model with one dimension predictor.}
\label{Figure0}
\end{center}
\end{figure}




%
%
%
%

\newpage

\setcounter{page}{1}
\renewcommand{\thesection}{S-\arabic{section}}
\setcounter{section}{0}

\begin{center}
\Huge{\emph{Supplementary Material}}
\end{center}

\section{Additional discussion on the practical aspects}

We proved in the paper that the choice of $\gamma_0^{(p)}$ is irrelevant for the asymptotic theory. In practice the choice of $\gamma_0^{(p)}$ could be related to prior information on a class of alternatives and thus help the practitioner to build a powerful test against that class. Concerning the set $B_p$, since $Q_n(\gamma) = Q_n(- \gamma)$ for any $\gamma \in \mathcal{S}^p,$ one could restrict the set $B_p$ to a half unit hypersphere like $\{\gamma\in \mathcal{S}^p: \gamma_1 \geq 0\} .$ One could restrict $B_p$ even more, and hence to speed up the optimization algorithms, when some prior information indicates a set of directions that would be able to detect alternatives.

Let us next discuss the influence of the choice of bandwidth. In our theory the choice of bandwidth does not appear in the asymptotic approximation of the size of the test. With finite samples, however, the law of $nh^{1/2}Q_n(\gamma)/\widehat v_n (\gamma)$ may change significantly with $h$ even for a fixed $\gamma,$ and hence a size correction is often necessary. We propose to make this correction using the simple wild bootstrap procedure described in the paper.  Alternatively, one could look for more elaborate methods, as for instance those in Horowitz and Spokoiny (2001) or Gao and Gijbels  (2008).
Such theoretical investigations could likely be reproduced in our framework under suitable, though restrictive, assumptions.
We argue that the  aspects concerning the choice of bandwidth in a functional data framework,  are quite challenging and hence deserve a separate investigation to be undertaken in future work.

\section{Additional empirical evidence}

Some additional experiments are provided to explore the possible influence of the algorithm on the statistical properties of the proposed test, the effect of the bandwidth and the penalization $\alpha_n$, and to study the possible approximation by the asymptotic distribution of the test statistic. As in the main part of the manuscript, all the results  will be provided on the basis of one thousand original samples.

\subsection{The sequential algorithm}

In Section 3.5 of the paper, a simplified algorithm for optimization in the hypersphere $\mathcal{S}^p$ is proposed which is based on $(p-1)$ one-dimensional optimizations. We shall show here that the statistical properties of the test with this simplified algorithm are similar to those obtained with a grid in the hypersphere. The same functional linear model studied in Section 4.2 of the paper will be considered. The procedure will be applied with the same values of the parameters used there which yielded the results in Tables 2 and 3. The dimension $p$ is now taken to be $3$ and $5$. Table S.1 contains the percentages of rejections for the new test with the sequential algorithm and with a full-dimensional optimization algorithm based on a grid in the hypersphere $\mathcal{S}^p$. For the sequential algorithm a grid of 50 points was taken in each one-dimensional optimization. For the full-dimensional  optimization, a grid of 1200 points was taken in the hypersphere  when $p=3$, while a grid of 3125 points was taken when $p=5$. It can be observed that similar results are obtained with the two algorithms, while the one-dimensional algorithm is much less time consuming and it is still feasible for a large dimension $p$.

\medskip
\begin{center}
\emph{Insert Table S.1 here}
\end{center}
\medskip


\subsection{The bandwidth}

To show the effect of the bandwidth on the test, the functional linear model studied in Section 4.2 is again considered, and the test is applied with the same values of the parameters used there. We shall fix the dimension $p$ here, to be 3, while the bandwidth will be $h=c_h n^{-2/9}$ for different values of $c_h$.

Table S.2 contains the percentages of rejections under the null hypothesis and under the alternative, coming from the functional linear effect. The level is respected for all values of the bandwidth, while the power is not much affected by the bandwidth in the wide range of values from $c_h=0.5$ to $c_h=1.5$. A possible trend to a higher power for larger bandwidths can be derived within this range, which is explained by the smoothness of the functional linear alternative. On the other hand, more curved alternatives generally require a smaller bandwidth as  was observed in similar smoothed testing methods. Either way, the test proposed here does not show a large effect coming from the bandwidth, and it should also be noted that, due to the nearest neighbor methodology, the choice of bandwidth does not depend on the covariate scale, so general rules like $h=c_h n^{-2/9}$, with $c_h$ around o1, are applicable to the common models considered in the literature on lack-of-fit tests.

\medskip
\begin{center}
\emph{Insert Table S.2 here}
\end{center}
\medskip


\subsection{The penalization, $\alpha_n$}

The penalization, denoted by $\alpha_n$, is one parameter that has to be chosen for the proposed test. We can say that there is no optimal choice for this parameter, but the decision will simply be based on the certainty about the best direction to detect the alternative. If the practitioner has a clear intuition that one direction will reveal a deviation from the null hypothesis, then this privileged direction should be protected by means of a large penalization $\alpha_n$. On the other hand, if there is no clear reference for a privileged direction, then a low penalization would be the most natural option. Regarding the values for the penalization, since the test statistic is asymptotically standard normal, values from 2 to 5 provide a balance between the privileged direction and the direction maximizing the deviation.

In our simulation results, the power (percentages of rejections), obtained for a penalization $\alpha_n$,  takes values from the power obtained with the privileged direction ($\alpha_n=\infty$) to the power obtained with the maximizing direction ($\alpha_n=0$). The directions we considered are the first eigenvector in the empirical FPC of the covariate, the second eigenvector and an un-informative direction with the same coefficients in all FPC components. The first eigenvector is a direction with very good power, the second has very poor power and the uninformative direction has  moderate power, similar to the maximizing direction.

Table S.3 shows the percentages of rejections with the same model studied in Section 4.2 of the paper, under the null hypothesis and under the alternative given by the functional linear model, with the second eigenvector of the empirical FPC of the covariate as privileged direction, and different values of the penalization. All other parameters and configurations of the test are the same as used in Section 4.2 of the paper. We have chosen the worst direction as the privileged direction, in order to show that some values of the penalization protect against a bad choice of the privileged direction. At the same time, this kind of privileged direction allows one to distinguish more clearly the balancing effect of the penalization between the privileged direction and the direction maximizing the statistic.

We observe that, under the null hypothesis, the nominal levels are respected for any value of the penalization. Under the alternative, the value 0 of the penalization provides the power coming from the direction maximizing the statistic, while the value $\infty$ corresponds to the test based only on the second eigenvector as privileged direction, which is a test with no power at all. It should be noted that the power is preserved with the values in the range from 0 up to  4 or 5, which we recommend. Smaller values, such as 1, 2 or 3, lead to a test based on the maximizing direction, while values from 4 or 5 lead to a test mainly based on the privileged direction.

\medskip
\begin{center}
\emph{Insert Table S.3 here}
\end{center}
\medskip

\subsection{Approximation by the asymptotic distribution}

Since the limit distribution of the test has been shown to be standard normal, one could wonder about the possibility of approximating the critical values of the test by the quantiles of the standard normal distribution.

The same model studied in Section 4.2 of the paper will be used under the null hypothesis, in the following experiment. In this way we shall check the accuracy of the standard normal approximation of the test. All the parameters and configurations of the test will be as in Section 4.2. The dimension $p$ will be set to 3, while different values for the penalization $\alpha_n$ will be considered.

The penalization will play a crucial role in the approximation of the test. A large value of the penalization leads to a test based on the projections of the covariate on the privileged direction, while a small value of the penalization leads to a test based on the maximum over the set of directions. Large values of the penalization then lead to smaller values of the test statistic, while small values lead to larger test statistics.

Table S.4 below shows the percentages of rejections for the test, when the standard normal distribution is used for the approximation of the critical values. Different nominal levels, sample sizes and values of the penalization, $\alpha_n$, are considered. Small values of the penalization (smaller than 2) will generally lead to percentages of rejections higher than the nominal level. Large values of the penalization will lead to percentages of rejections smaller than the nominal level. This fact is due to the negative correlation of the residuals which makes the test statistic based on a single direction negatively biased for small sample sizes. Since the approximation of the standard normal distribution would only be valid for a fine tuning of the penalization, we generally propose to use the bootstrap approximation that was shown to work for all models considered.

\medskip
\begin{center}
\emph{Insert Table S.4 here}
\end{center}
\medskip


\newpage

{\small
$$\begin{tabular}{ccccccccccc}
\vspace*{-1mm}
&&&\multicolumn{2}{c}{Level\ $=10\%$} && \multicolumn{2}{c}{Level\ $=5\%$}
&& \multicolumn{2}{c}{Level\ $=1\%$} \\
Hypothesis & $p$ & Algorithm & $n=40$ & $n=100$ && $n=40$ & $n=100$ && $n=40$ & $n=100$ \\
\hline &&&&&&&&&& \\
Null
  & 3 & Sequential &  9.6 & 10.7 && 5.2 & 5.7 && 1.0 & 1.3 \\ \vspace*{1.5mm}
  &   &   Full-dim. &  9.8 & 11.0 && 4.9 & 6.4 && 1.1 & 1.4 \\
  & 5 & Sequential & 10.1 & 10.4 && 5.2 & 5.9 && 0.9 & 1.1 \\ \vspace*{1.5mm}
  &   &   Full-dim. & 10.2 & 11.5 && 5.4 & 6.3 && 1.2 & 1.6 \\
Alternative
  & 3 & Sequential & 53.3 & 94.8 && 42.1 & 89.7 && 20.1 & 75.9 \\ \vspace*{1.5mm}
  &   &   Full-dim. & 50.7 & 93.0 && 40.5 & 88.3 && 20.1 & 74.7 \\
  & 5 & Sequential & 50.4 & 92.2 && 38.5 & 87.6 && 16.5 & 71.3 \\
  &   &   Full-dim. & 54.2 & 92.3 && 42.0 & 88.8 && 20.4 & 74.3 \\
 \hline
\end{tabular}$$
\begin{center}
{\bf Table S.1.} Percentages of rejections for the new test with the sequential algorithm and with a full-dimensional grid-based algorithm.
\end{center}
}

\qquad

\qquad

\qquad

{\small
$$\begin{tabular}{cccccccccc}
\vspace*{-1mm}
&&\multicolumn{2}{c}{Level\ $=10\%$} && \multicolumn{2}{c}{Level\ $=5\%$}
&& \multicolumn{2}{c}{Level\ $=1\%$} \\
Hypothesis & $c_h$ & $n=40$ & $n=100$ && $n=40$ & $n=100$ && $n=40$ & $n=100$ \\
\hline &&&&&&&&& \\
Null &  0.5  &  11.8 & 11.1 &&  6.0 &  6.2 && 1.3 & 1.1 \\
     &  0.75 &  11.2 & 10.7 &&  4.9 &  5.4 && 1.1 & 1.0 \\
     &  1.0  &   9.6 & 10.7 &&  5.2 &  5.7 && 1.0 & 1.3 \\
     &  1.25 &   9.2 & 10.9 &&  5.3 &  6.0 && 0.7 & 1.1 \\ \vspace*{1.5mm}
     &  1.5  &   9.3 &  9.9 &&  5.1 &  5.5 && 0.8 & 1.2 \\
Alternative
     &  0.5  &  45.4 & 89.4 && 34.7 & 84.3 && 16.1 & 68.7 \\
     &  0.75 &  48.1 & 92.2 && 37.4 & 86.6 && 18.7 & 72.6 \\
     &  1.0  &  53.3 & 94.8 && 42.1 & 89.7 && 20.1 & 75.9 \\
     &  1.25 &  54.0 & 95.8 && 42.0 & 91.5 && 21.0 & 78.2 \\
     &  1.5  &  66.9 & 97.4 && 43.1 & 93.5 && 20.4 & 79.6 \\
 \hline
\end{tabular}$$
\begin{center}
{\bf Table S.2.} Percentages of rejections for the new test under the null hypothesis and the alternative, for different values of the bandwidth $h=c_h n^{-2/9}$.
\end{center} }

\qquad

\qquad

\qquad

{\footnotesize
$$\begin{tabular}{ccccccccccc}
\vspace*{-1mm}
&&\multicolumn{2}{c}{Level\ $=10\%$} && \multicolumn{2}{c}{Level\ $=5\%$}
&& \multicolumn{2}{c}{Level\ $=1\%$} \\
Model & $\alpha_n$ & $n=40$ & $n=100$ && $n=40$ & $n=100$ && $n=40$ & $n=100$ \\
\hline &&&&&&&& \\
Null
 & 0.0  & 9.9 & 10.3 && 4.7 & 5.0 && 1.0  & 1.3  \\
 & 0.5  & 9.6 & 10.3 && 4.9 & 5.2 && 1.1  & 1.3  \\
 & 1.0  & 9.5 & 10.2 && 4.9 & 5.5 && 1.1  & 1.4  \\
 & 1.5  & 9.2 &  9.7 && 4.9 & 5.5 && 1.1  & 1.5  \\
 & 2.0  & 9.7 & 10.0 && 4.9 & 5.7 && 1.1  & 1.5  \\
 & 2.5  & 9.7 & 10.7 && 4.7 & 5.6 && 1.1  & 1.3   \\
 & 3.0  & 8.5 & 10.6 && 4.2 & 5.3 && 1.0  & 1.2  \\
 & 3.5  & 8.9 & 11.5 && 4.5 & 5.6 && 1.1  & 1.3  \\
 & 4.0  & 9.3 & 11.1 && 4.9 & 5.9 && 1.0  & 1.2  \\
 & 4.5  & 9.3 & 10.7 && 4.8 & 5.6 && 1.0  & 1.2  \\
 & 5.0  & 9.4 & 10.8 && 4.8 & 6.1 && 1.1  & 0.9  \\
 & 7.0  & 9.2 & 10.9 && 4.7 & 6.0 && 0.8  & 1.1  \\
 &10.0  & 9.2 & 10.7 && 4.7 & 5.8 && 0.9  & 0.9  \\ \vspace*{1mm}
 & $\infty$ & 9.2 & 10.7 &&  4.7 &  5.8 &&  0.9 &  0.9 \\
Alternative
 & 0.0  & 65.1 & 98.6 && 51.9 & 96.2 && 23.9  & 85.4  \\
 & 0.5  & 65.5 & 98.6 && 52.4 & 96.2 && 24.1  & 85.5  \\
 & 1.0  & 66.0 & 98.6 && 52.9 & 96.2 && 24.4  & 85.6  \\
 & 1.5  & 67.0 & 98.5 && 53.3 & 96.2 && 24.7  & 85.8  \\
 & 2.0  & 68.2 & 98.6 && 54.3 & 96.4 && 25.2  & 86.4  \\
 & 2.5  & 66.1 & 98.6 && 54.6 & 96.6 && 26.0  & 86.3  \\
 & 3.0  & 56.8 & 96.1 && 51.5 & 96.5 && 26.3  & 86.6  \\
 & 3.5  & 48.2 & 96.4 && 43.9 & 95.3 && 26.8  & 87.1  \\
 & 4.0  & 39.7 & 93.3 && 36.4 & 92.2 && 27.2  & 87.0  \\
 & 4.5  & 30.7 & 89.5 && 26.9 & 88.0 && 22.7  & 85.4  \\
 & 5.0  & 23.6 & 83.3 && 19.3 & 81.4 && 15.7  & 79.7  \\
 & 7.0  & 10.0 & 51.0 &&  5.6 & 47.6 &&  2.1  & 44.6  \\
 &10.0  &  8.6 & 21.5 &&  4.0 & 16.4 &&  0.5  & 12.6 \\
 & $\infty$ &  8.6 &  10.9 &&  4.0 &  5.2 &&  0.5 &  1.3 \\
\hline
\end{tabular}$$
\begin{center}
{\bf Table S.3.} Percentages of rejections for the new test under the null hypothesis and under the alternative (functional linear deviation), for different values of the penalization, $\alpha_n$.
\end{center} }

\qquad

\qquad

\qquad

{\small
$$\begin{tabular}{ccccccccc}
\vspace*{-1mm}
&\multicolumn{2}{c}{Level\ $=10\%$} && \multicolumn{2}{c}{Level\ $=5\%$}
&& \multicolumn{2}{c}{Level\ $=1\%$} \\
$\alpha_n$ & $n=40$ & $n=100$ && $n=40$ & $n=100$ && $n=40$ & $n=100$ \\
\hline &&&&&&&& \\
  0.0  & 14.6 & 21.7 &&  8.8 & 15.3 &&  4.1 &  6.5 \\
  0.5  & 14.4 & 21.2 &&  8.8 & 15.1 &&  4.1 &  6.4 \\
  1.0  & 13.8 & 19.8 &&  8.5 & 14.2 &&  3.8 &  6.4 \\
  1.5  & 12.4 & 18.0 &&  8.0 & 13.6 &&  3.7 &  6.4 \\
  2.0  & 10.7 & 15.2 &&  7.6 & 12.1 &&  3.7 &  5.9 \\
  2.5  &  8.0 & 12.2 &&  6.6 & 10.1 &&  3.4 &  5.5 \\
  3.0  &  5.5 &  7.9 &&  5.3 &  7.1 &&  2.9 &  4.6 \\
  3.5  &  2.9 &  5.0 &&  2.8 &  4.6 &&  2.1 &  3.7 \\
  4.0  &  2.0 &  3.3 &&  1.9 &  2.9 &&  1.4 &  2.3 \\
  4.5  &  1.2 &  2.2 &&  1.1 &  1.7 &&  0.6 &  1.0 \\
  5.0  &  1.0 &  1.9 &&  0.9 &  1.4 &&  0.4 &  0.7 \\
  7.0  &  0.8 &  1.7 &&  0.7 &  1.2 &&  0.2 &  0.5 \\
$\infty$ &  0.8 &  1.7 &&  0.7 &  1.2 &&  0.2 &  0.5 \\
\hline
\end{tabular}$$
\begin{center}
{\bf Table S.4.} Percentages of rejections for the new test under the null hypothesis, with asymptotic distribution approximation, for different values of the penalization, $\alpha_n$.
\end{center} }
\vspace*{3mm}

\newpage

\section{Technical lemmas and proofs}

\begin{lem}\label{db_integr}
Let $K$ be a density satisfying Assumption \ref{K}-(a) and assume that $h \rightarrow 0$ and $nh\rightarrow \infty$. Let
$$
S_{ni} = \frac{1}{(n-1)h}\sum_{1\leq j \leq n, \; i\neq j} K\left( \frac{i-j}{nh} \right) \quad \text{and} \quad S_n=\frac{1}{n}\sum_{1\leq i \leq n} S_{ni}.
$$
Constants $c_1,c_2$ then exist such that for sufficiently large $n$
$$
0 < c_1\leq \min_{1\leq i \leq n } S_{ni}\leq  \max_{1\leq i \leq n } S_{ni}  \leq c_2 < \infty.
$$
Moreover,  $S_n \rightarrow 1.
$
\end{lem}

\begin{proof}[Proof of Lemma \ref{db_integr}]
It is clear that $S_n - \widetilde S_n \rightarrow 0$ where
$$
\widetilde S_n = \frac{1}{n^2h}\sum_{1\leq i, j \leq n} K\left( \frac{i-j}{nh} \right) .
$$
If $[a]$ denotes the integer part of any real number $a$,  we can write
\begin{eqnarray*}
\widetilde S_n &=& \int_{1/n}^{(n+1)/n} \int_{1/n}^{(n+1)/n} h^{-1} K\left( \frac{[ns] - [nt]}{nh} \right) dsdt\\
&=& \int_{1/n}^{(n+1)/n} \int_{1/nh - t/h }^{1/h + 1/nh - t/h}  K\left( \frac{[nt + nzh] - [nt]}{nh} \right) dzdt\qquad \qquad \quad\; [z=(s-t)/h]\\
&=& \int_{1/n}^{(n+1)/n} \int_{1/nh - t/h }^{1/h + 1/nh - t/h} K\left( z\right) dzdt + o(1) \\
&=& \int_{- 1/h}^{1/h} \int_{1/n - zh }^{1+1/n - zh} dt K\left( z\right) dz + o(1) \qquad \qquad \qquad \qquad \qquad \qquad \qquad \qquad \![\text{Fubini}]\\
&\rightarrow & 1,
\end{eqnarray*}
where the order $o(1)$ of the remainder on the right-hand side of  the third equality could be obtained as a consequence of the fact that  $K$ is symmetric and monotonic.
Hence  $S_n\rightarrow 1.$ Similarly, we can write
\begin{eqnarray*}
\widetilde S_{ni} &=& \int_{1/n}^{(n+1)/n} h^{-1} K\left( \frac{i - [nt]}{nh} \right) dt\\
&=&  \int_{(1-i)/nh}^{1/h + (1-i)/nh}  K\left( \frac{i - [i + nzh] }{nh} \right) dz\qquad \qquad \quad\; [z=(t-i/n)/h]\\
&=&  \int_{(1-i)/nh }^{1/h + (1-i)/nh } K\left( z\right) dz + o(1).
\end{eqnarray*}
Deduce that
$$
\int_{0}^1 K(z) dz + \underline{r}_{ni} \leq \widetilde S_{ni} \leq \int_{\mathbb{R}} K(z) dz + \overline{r}_{ni}
$$
where $\max_{1\leq i \leq n } \{|\underline r_{ni}|+|\overline r_{ni}|\} = o(1).$ The result follows.
\end{proof}

\bigskip

\bigskip

\begin{proof}[Proof of Lemma \ref{upp_bds}]
The bound for $A_n$ is obvious. For $C_n^2$ note that
$$
\mathbb{E}[ h_{i,j}^2 (Z_i,Z_j)] = \frac{ M^{-4}}{n^2 (n\!-\!1)^2 h} \mathbb{E}
\! \left[ \langle  Z_{i} \; , Z_{j}   \rangle^2 \right]\! h^{-1} K_{h}^2\left( F_{\gamma,n}(\langle x_{i},\gamma\rangle)- F_{\gamma,n}(\langle x_{j},\gamma\rangle) \right) .
$$
By the Cauchy-Schwarz inequality and the triangle inequality and recalling that $\widetilde Z_i$
is distributed according to the conditional law of $U_i$ given $X_i=x_i$,
$$
\mathbb{E}
\left[ \langle  Z_{i} \; ,  Z_{j}   \rangle^2 \right] \leq 16 \mathbb{E}
\left[ \| \widetilde Z_{i} \|^2  \right]
\mathbb{E}
\left[  \| \widetilde Z_{j} \|^2  \right]\leq 16 C^2,
$$
for any constant  $C$  that bounds from above $\mathbb{E}(\|U\|^2\mid X)$, see Assumption \ref{D}-(c).
Finally, recall that $K$ is bounded and note that
$$
\frac{1}{n(n-1)h}\sum_{1\leq i\neq j \leq n} K_h\left( F_{\gamma,n}(\langle x_{i},\gamma\rangle)- F_{\gamma,n}(\langle x_{j},\gamma\rangle) \right) = \frac{1}{n(n-1)h}\sum_{1\leq i\neq j \leq n} K\left( \frac{i-j}{nh} \right)
$$
and apply the second part of Lemma \ref{db_integr} to derive the bound for $C_n^2.$ To derive the bound for $B_n^2$ recall that $h_{i,j}(Z_j,z)  $ vanishes for $\|z\| > 2M$, again using the Cauchy-Schwarz inequality and the triangle inequality and the first part of  Lemma \ref{db_integr}. For the bound of $D_n$, using the Cauchy-Schwarz inequality and the independence of $Z_i$ and $Z_j,$ we can write
\begin{eqnarray*}
\mathbb{E} \!\sum_{i,j} \! h_{ i, j}(Z_i,\!Z_j)f_i(Z_i) g_j(Z_j) \!\!  &\leq&  \!\!\! \sum_{i,j} \frac{\mathbb{E} |\langle  Z_{i}f_i(Z_i) ,  Z_{j}g_j(Z_j)   \rangle |}{n(n-1)hM^2}
K_{h}\!\left( F_{\gamma,n}(\langle x_{i},\gamma\rangle)\!- \!F_{\gamma,n}(\langle x_{j},\gamma\rangle) \right) \\
&\leq & \!\!\! \sum_{i,j} \!\frac{\!4C \mathbb{E} ^{1/2}\!f_i^2(Z_i) \mathbb{E} ^{1/2}  \!g_j^2(Z_j)}{n(n-1)hM^2}
K_{h}\! \left( F_{\gamma,n}(\langle x_{i},\gamma\rangle)\!- \!F_{\gamma,n}(\langle x_{j},\gamma\rangle) \right) \\
&\leq & \frac{\!4C}{M^2} \|| \mathcal{K} \||_2,
\end{eqnarray*}
where $C$ is such that $\mathbb{E}(\|U\|^2\mid X)\leq C$ and
$\mathcal{K} $ is the matrix with elements
\begin{equation}\label{matrix_k}
\mathcal{K}_{ij}=  K\left( (i-j)/nh \right)/[n(n-1)h], \quad i\neq j, \quad \text{and} \quad
\mathcal{K}_{ii}= 0,
\end{equation}
and $\|| \mathcal{K} \||_2$  is the spectral norm of $\mathcal{K}.$
By definition, $\|| \mathcal{K} \||_2 =\sup_{u\in\mathbb{R}^n , u\neq 0}{\Vert
\mathcal{K} u\Vert }/{\Vert u\Vert }$ and $|u^\prime \mathcal{K} w|\leq \|| \mathcal{K} \||_2 \|u\| \|w\|$ for any $u,w\in\mathbb{R}^n.$
By Lemma \ref{db_integr}, for any $u\in
\mathbb{R}^{n}$,
\begin{eqnarray}
\left\Vert \mathcal{K} u\right\Vert ^{2} &=& \sum_{i=1}^{n}\left(
\sum_{j=1,j\neq i}^{n}\frac{K_{h}\left( (i-j)/nh \right) }{h\,n(n-1)}u_{j}\right) ^{2}\notag \\
&\leq &\sum_{i=1}^{n}\left( \sum_{j=1,j\neq i}^{n}\frac{K_{h}\left(
(i-j)/nh\right) }{h\,n(n-1)}\right)
\sum_{j=1,j\neq i}^{n}\frac{K_{h}\left( (i-j)/nh\right) }{h\,n(n-1)}\,u_{j}^{2} \notag \\
&\leq &\left\Vert u\right\Vert ^{2}\left[ \max_{1\leq i\leq n}\left(
\sum_{j=1,j\neq i}^{n}\frac{K_{h}\left( (i-j)/nh\right) }{h\,n(n-1)}\right) \right] ^{2}\;\notag\\
&\leq &c n^{-2}\|u\|^2,\label{spectral}
\end{eqnarray}
for some constant $c>0$. The bound for $D_n$ follows immediately.
\end{proof}

\bigskip

\begin{proof}[Proof of Theorem \ref{as_law}]
By Lemma \ref{beta}, if suffices to prove the asymptotic normality of the test statistic $T_n$ defined with $\widehat \gamma_n =\gamma_0^{(p)}.$ The proof of this asymptotic normality is based on
the Central Limit Theorem 5.1 of de Jong (1987). We will apply the result of de Jong  conditionally given the values of the
covariate sample.
Let $x_1,\cdots,x_n$ be an \emph{arbitrary} collection of non-random points in $L^2[0,1].$
Consider $\widetilde Z_1,\cdots, \widetilde Z_n$ independent random variables with values in $L^2[0,1]$ such that for each $i,$ the law of $\widetilde Z_i$
is the conditional law of $U_i$ given $X_i=x_i$.
Let $F_{\gamma_0^{(p)},n}(\cdot)$ be the empirical distribution function of the sample $\langle x_1, \gamma_0^{(p)}\rangle,\cdots, \langle x_n, \gamma_0^{(p)}\rangle$,
$$K_{h,ij}(\gamma_0^{(p)}) =  K_{h}\left( F_{\gamma_0^{(p)},n}(\langle x_{i},\gamma_0^{(p)}\rangle)- F_{\gamma_0^{(p)},n}(\langle x_{j},\gamma_0^{(p)}\rangle) \right)$$ and
$$
W_{ij} = \frac{1}{n(n-1)}\langle \widetilde Z_{i} , \widetilde Z_{j} \rangle \frac{1}{h} K_{h,ij}(\gamma_0^{(p)})
,\quad 1\leq i\neq j\leq n, \qquad W_{ii} =0, \quad 1\leq i\leq n.
$$
Hence $Q_{n}(\gamma_0^{(p)} ) =\sum_{i,
j} W_{ij}$ and $\widehat{v}_{n}^{2}( \gamma_0^{(p)}) = 2n(n-1)h \sum_{i,
j} W_{ij}^2.$ A crucial remark that is used several times in the following is that the elements of the matrix $(K_{h,ij}(\gamma_0^{(p)}))$ are the same as those of the matrix $(K_{h}((i-j)/nh)$ up to permutations of lines and columns.
Following the notation of de Jong (1987), let $$\sigma^2_{ij} = \mathbb{E}(W_{ij}^2)= \mathbb{E}[\langle  U_{i} ,  U_{j} \rangle^2 \mid X_i=x_i,X_j = x_j] \frac{K_{h,ij}^2(\gamma_0^{(p)})}{n^2(n-1)^2h^{2}} $$ and $\sigma(n)^2 = 2 \sum_{i\neq j} \sigma^2_{ij}.$
Since $$ \mathbb{E}[ \langle U_{i} , U_{j} \rangle ^2 \mid X_1,\cdots, X_n] =  \mathbb{E}[ \langle U_{i} , U_{j} \rangle ^2 \mid X_i, X_j]
\leq \mathbb{E}[ \| U_{i} \|^2 \mid X_i]  \mathbb{E}[ \| U_{j} \|^2 \mid X_j],$$ and $\mathbb{E}[ \langle U_{i} , U_{j} \rangle ^2 \mid X_i, X_j]$ is bounded away from zero by Assumption \ref{D}-(c), deduce that positive constants $\underline{c}$ and $\overline{c}$ exist such that
\begin{equation}\label{sigmaij}
\frac{\underline{c}}{n^4 h^2}  K^2_{h,ij}(\gamma_0^{(p)}) \leq \sigma_{ij}^2 \leq \frac{\overline{c}}{n^4 h^2}   K^2_{h,ij}(\gamma_0^{(p)}).
\end{equation}
Applying Lemma \ref{db_integr} with $K$ replaced by $K^2,$ one can deduce that for each $i,$
\begin{eqnarray}\label{rty}
\frac{c_1}{n^3 h} \leq \frac{\underline{c}}{n^4 h^2} \min_{1\leq i \leq n} \sum_{ j\neq i } K^2_{h}((i-j)/nh)&\leq &  \sum_{1\leq  j \leq n, i\neq j } \sigma_{ij}^2 \\&\leq &\frac{\overline{c}}{n^4 h^2} \max_{1\leq i \leq n} \sum_{ j\neq i } K^2_{h}((i-j)/nh) \leq \frac{c_2}{n^3 h},\notag
\end{eqnarray}
for some constants $c_1$ and $c_2.$ Moreover,
constants $\underline{c}^\prime$ and $\overline{c}^\prime$ exist such that
$$
\underline{c}^\prime n^{-2}h^{-1} \leq \sigma(n)^2\leq \overline{c}^\prime n^{-2}h^{-1}.
$$
It follows that
$$
\sigma(n)^{-2} \max_{1\leq i\leq n} \sum_{j=1}^n \sigma_{ij}^2 = O(n^{-1}),
$$
and thus Condition 1 in Theorem 5.1 of de Jong (1987) holds true as soon as $\kappa (n) = o(n^{1/2}).$ In order to check Condition 2 in Theorem 5.1 of de Jong (1987), let us use the H\"{o}lder inequality with $p=\nu/2$ and $q=\nu/(\nu -2),$ with $\nu$ given by Assumption \ref{D}-(c)-(ii), and the Markov inequality to get, for some constant $C$,
$$
 \mathbb{E}[\sigma_{ij}^{-2}W_{ij}^2 \mathbb{I}_{\{\sigma^{-1}_{ij}|W_{ij}| > \kappa(n)  \} } ] \leq    \mathbb{E}^{2/\nu}[\sigma_{ij}^{-\nu}|W_{ij}|^\nu] \mathbb{P}^{(\nu-2)/\nu}
 [\sigma^{-1}_{ij}|W_{ij} |> \kappa(n)] \leq C \kappa(n)^{-(\nu-2)/\nu} .
$$
This shows that Condition 2 of Theorem 5.1 of de Jong holds true for any $\kappa (n)$ tending to infinity.  Finally, let $\mu_1,\cdots,\mu_n$ denote the eigenvalues of the matrix $(\sigma_{ij}).$ To check Condition 3 of de Jong, we use the upper bound of $\sigma_{ij}$ in (\ref{sigmaij}) to deduce that there exists a constant $C$ (independent of $n$ and $i$) such that
\begin{eqnarray*}
\sum_{j=1,j\neq i}^{n}\sigma_{ij} &\leq &\frac{C}{n^2 h} \sum_{j=1,j\neq i}^{n}K_{h,ij}(\gamma_0^{(p)}).
\end{eqnarray*}
Next, it should be noted that if $\Sigma$ denotes the $n\times n$ matrix with generic element $\sigma_{ij},$ following the lines of equation (\ref{spectral}) and using equation (\ref{sigmaij}),
for any $u\in\mathbb{R}^{n}$,
\begin{eqnarray}\label{eig_mu}
\left\Vert \Sigma u\right\Vert ^{2}
&\leq &\left\Vert u\right\Vert ^{2}\left[ \max_{1\leq i\leq n}\left(
\sum_{j=1,j\neq i}^{n}\sigma_{ij}\right) \right] ^{2}\notag\\
&\leq & c_1 \left\Vert u\right\Vert ^{2}\left[ \max_{1\leq i\leq n}\left(
\sum_{j=1,j\neq i}^{n}\frac{K_{h}\left( (i-j)/nh\right) }{h\,n(n-1)}\right) \right] ^{2}
\leq c_2 n^{-2}\|u\|^2,
\end{eqnarray}
for some constants $c_1,c_2>0$. It can be deduced that
$$
\sigma(n)^{-2} \max_{1\leq i\leq n}\mu_i^2 \leq \frac{h n^2}{\underline{c}^\prime } \frac{c_2}{n^2} \rightarrow 0,
$$
and thus Condition 3 of de Jong (1987) holds true.
To complete the proof of the asymptotic normality of the statistic $T_{n} = n h^{1/2} Q_{n}(\gamma_0^{(p)})/
\widehat{ v}_{n} (\gamma_0^{(p)})$ given the covariate values, it should be noted that
$$
\sigma^2(n)  = \mathbb{E}[Q_{n}^2(\gamma_0^{(p)} ) \mid X_1=x_1,\cdots,X_n = x_n ]= \frac{\mathbb{E}[\widehat{v}_{n}^{2}( \gamma_0^{(p)})\mid X_1=x_1,\cdots,X_n = x_n]}{n(n-1)h}.
$$
Moreover, by direct standard calculation, it can be shown that the variance of
$$
 \frac{1}{n(n-1)}\sum_{1\leq i\neq j \leq n} \langle \widetilde Z_{i} , \widetilde Z_{j} \rangle^2  \frac{1}{h} K^2_{h,ij}(\gamma_0^{(p)})
$$
is of rate $O(h^{-1}n^{-1})=o(1)$.
One can deduce that
\begin{equation}\label{var_esty}
\frac{\widehat{v}_{n}^{2}( \gamma_0^{(p)})/n(n-1)h}{\sigma^2(n)} - 1 = o_{\mathbb{P}} (1)
\end{equation}
given $ X_1=x_1,\cdots,X_n = x_n.$ The asymptotic normality of $T_n$ given $ X_1=x_1,\cdots,X_n = x_n$ is a consequence of Theorem 5.1 of de Jong and equation (\ref{var_esty}). The proof is complete.
\end{proof}

\bigskip

\bigskip
\begin{proof}[Proof of Theorem \ref{altern}]
The proof is based on inequality (\ref{eqaa}).
Since $\mathbb{E}(\langle U_{1},U_{2}\rangle^{2} \mid X_{1}, X_{2}) \geq \underline{\sigma}^2 + r_n^4\langle \delta(X_{1}),\delta(X_{2})\rangle^{2},$ clearly the variance estimate $\widehat v_n^2 (\widetilde \gamma)$ does not approach zero for all $\widetilde \gamma$. On the other hand, by Cauchy-Schwarz and the property of the spectral norm for matrices,
\begin{eqnarray}\label{spec_ineqp}
\widehat v_n^2 (\widetilde \gamma) &\leq &\frac{2n/(n-1)}{n^2h}\sum_{1\leq i,j\leq n}\|\delta(X_i)\|^2\|\delta(X_j)\|^2 K_h^2(F_{n,\widetilde \gamma}(\langle X_i ,\widetilde \gamma\rangle) - F_{n,\widetilde \gamma}(\langle X_j ,\widetilde \gamma\rangle))\notag\\
&\leq & \|| \mathcal{K}_2 \||_2 \sum_{1 \leq i \leq n} \|\delta(X_i)\|^4,
\end{eqnarray}
where $\mathcal{K}_2$ is the matrix with entries $n^{-2}h^{-1}K_h^2(F_{n,\widetilde \gamma}(\langle X_i ,\widetilde \gamma\rangle) - F_{n,\widetilde \gamma}(\langle X_j ,\widetilde \gamma\rangle)).$ From the arguments used in equation (\ref{eig_mu}),
$\|| \mathcal{K}_2 \||_2 = O_{\mathbb{P}}(n^{-1})$. This together with the finite fourth order moment condition for $\delta(\cdot)$ imply
that $\widehat v_n^2 (\widetilde \gamma)$ is bounded in probability.
Hence it suffices to examine at the behavior of $Q_n(\widetilde \gamma)$. By Lemma \ref{lem1}-(B) there exists $p_0$  and $\widetilde \gamma\in B_{p_0}\subset \mathcal{S}^{p_0}$ ($p_0$ and $\widetilde \gamma$ independent of $n$) such that $\mathbb{E}[\delta(X) \mid \langle X,\widetilde  \gamma\rangle]\neq 0$.  Hereafter, $\widetilde\gamma $ is supposed to have this property. Let $V_{ni} =F_{n,\widetilde \gamma}(\langle X_i ,\widetilde \gamma\rangle).$
We can write
\begin{eqnarray*}
Q_n(\widetilde \gamma) & = & \frac{1}{n(n-1)h}\sum_{i\neq j}\langle U_i^0 ,U_j^0\rangle K_h(V_{ni} - V_{nj}\rangle )\\
&&+ \frac{2r_n}{n(n-1)h}\sum_{i\neq j}\langle U_i^0 ,\delta(X_j)\rangle K_h(V_{ni} - V_{nj} )\\
&&+\frac{r_n^2}{n(n-1)h}\sum_{i\neq j}\langle \delta(X_i), \delta(X_j)\rangle K_h(V_{ni} - V_{nj} )\\
&=:& Q_{0n}(\widetilde \gamma) + 2r_n  Q_{1n}(\widetilde \gamma)+ r_n^2Q_{2n}(\widetilde \gamma).
\end{eqnarray*}
Since $\widetilde \gamma$ is fixed,  $Q_{0n}(\widetilde \gamma)=O_{\mathbb{P}}(n^{-1}h^{-1/2})$ (cf. proof of Theorem \ref{as_law}).
Next, let us follow Guerre and Lavergne (2005), and denote by $\mathbb{E}_n$ the conditional expectation given $X_1,\cdots,X_n$ and define
$$
\overline{\delta_n}(X_i) = \frac{1}{n(n-1)h} \sum_{j=1, \;j\neq i}^n \delta(X_j) K_h(V_{ni} - V_{nj}),\qquad \overline{\delta}= (\delta(X_1),\cdots,\delta(X_n))^\prime.
$$
Then the Marcinkiewicz-Zygmund inequality
and the Cauchy-Schwarz and Jensen   inequalities imply that
\begin{eqnarray*}
\mathbb{E}_n\left|\sum_{i=1}^n \langle U_i^0 ,\overline{\delta_n} (X_{i}) \rangle \right| &\leq & c \; \mathbb{E}_n \left| \sum_{i=1}^n \left| \langle U_i^0 ,\overline{\delta_n} (X_{i}) \rangle \right|^2 \right|^{1/2}\leq c\; \mathbb{E}_n\left| \sum_{i=1}^n \| U_i^0 \|^2 \| \overline{\delta_n} (X_{i}) \|^2 \right|^{1/2} \\
&\leq& \!\!\! c \!\left\{ \sum_{i=1}^n \mathbb{E}_n \left( \| U_i^0 \|^2\right) \| \overline{\delta_n} (X_{i})\|^2  \right\}^{1/2}\!\!\!\leq c\; C_2^{1/\nu}
\!\left\{ \sum_{i=1}^n  \| \overline{\delta_n} (X_{i})\|^2  \right\}^{1/2}\\
&=& c \;C_2^{1/\nu} \|\mathcal{K}_3\overline{\delta}\| \leq c \;C_2^{1/\nu} n^{1/2} \|| \mathcal{K}_3 |\|_2 \left\{ \frac{1}{n}\sum_{i=1}^n  \| \delta (X_{i})\|^2  \right\}^{1/2},
\end{eqnarray*}
for  $\mathcal{K}_3$ a matrix with the same elements as the matrix $\mathcal{K}$ defined  in equation (\ref{matrix_k}) up to permutations of lines and columns, and $C_2$ and $\nu$ the constants in Assumption \ref{D}, and c some constant  independent of $n$. Since $\|| \mathcal{K} |\|_2=\|| \mathcal{K}_3 |\|_2 = O_{\mathbb{P}}(n^{-1})$, one can deduce that $Q_{1n}(\widetilde \gamma)=O_{\mathbb{P}}(n^{-1/2})$ conditionally on $X_1,\cdots,X_n.$
Now, let us investigate $Q_{2n}(\widetilde \gamma)$.
With an inequality such as in equation (\ref{spec_ineqp}) and the moment conditions on $\delta(\cdot)$ it is easy to bound $Q_{2n}(\widetilde \gamma)$ in probability. It remains to show that it is bounded  away from zero.  Let
$V_i =F_{\widetilde \gamma}(\langle X_i ,\widetilde \gamma\rangle) $, so that $V_1,\cdots, V_n$ are independent uniform  variables on $[0,1],$ and
$$
Q^{\prime}_{2n}(\widetilde \gamma) = \frac{1}{n^2h}\sum_{1\leq i , j\leq n}\langle\delta(X_{i}),\delta(X_{j})\rangle K_h( V_{ni} - V_{nj} ),
$$
$$
Q^{\prime\prime}_{2n}(\widetilde \gamma) = \frac{1}{n^2 h}\sum_{1\leq i , j\leq n}\langle\delta(X_{i}),\delta(X_{j})\rangle K_h( V_i - V_j ),
$$
$$
Q^{\star}_{2n}(\widetilde \gamma) = \frac{1}{n(n-1) h}\sum_{1\leq i \neq j\leq n}\langle\delta(X_{i}),\delta(X_{j})\rangle K_h( V_i - V_j ).
$$
We have
$$
Q^{\prime}_{2n}(\widetilde \gamma) - \frac{n\!-\!1}{n} Q_{2n}(\widetilde \gamma) = Q^{\prime\prime}_{2n}(\widetilde \gamma) - \frac{n\!-\!1}{n} Q_{2n}^\star(\widetilde \gamma) \!=\!  \frac{K(0)}{n^2h}  \sum_{i=1}^n \| \delta(X_{i}) \|^2 \!=\! O_{\mathbb{P}}(n^{-1}h^{-1}) \!= \!o_{\mathbb{P}}(1).
$$
Next we show that $Q^{\prime}_{2n}(\widetilde \gamma) -
Q^{\prime\prime}_{2n}(\widetilde \gamma) =  o_{\mathbb{P}}(1).$
If $K$ satisfies a Lipschitz condition and $nh^4\rightarrow \infty$, by the Cauchy-Schwarz inequality, for some constant $C>0$
\begin{eqnarray*}
\left|Q^{\prime}_{2n}(\widetilde \gamma) - Q^{\prime\prime}_{2n}(\widetilde \gamma)\right| & \leq & \frac{C\Delta_n}{h^2} \left[\frac{1}{n}\sum_{1\leq i\leq n}\|\delta(X_{i})\| \right]^2 = o_{\mathbb{P}}(1),
\end{eqnarray*}
where $\Delta_n = \sup_{1\leq i \leq n} |V_{ni} - V_{i}|$.
Note that
$
\Delta_n \leq \sup_{t\in\mathbb{R}} |F_{n,\widetilde \gamma} (t)- F_{\widetilde \gamma} (t)|=O_{\mathbb{P}} (n^{ -1/2}).
$
One can conclude that  $Q_{2n}(\widetilde \gamma) - Q^\ast_{2n}(\widetilde \gamma)=o_{\mathbb{P}}(1),$ so that is suffices to investigate $Q^\ast_{2n}(\widetilde \gamma)$.
It is easy to show that the variance of $Q^\ast_{2n}(\widetilde \gamma)$ tends to zero, so that it remains to show that the expectation of
$Q^\ast_{2n}(\widetilde \gamma) $ does not approach zero.
Let $\overline \delta (t,v) = \mathbb{E}[\delta(X_j)(t)\mid V_j=v]$ and note that $0<\iint_{[0,1]\times [0,1]} |\overline \delta(t,v) |^2dvdt<\infty.$
By the Inverse Fourier Transform formula and repeated applications of Fubini's theorem we get
\begin{eqnarray*}
\mathbb{E}[Q^\ast_{2n}(\widetilde \gamma)] &=& \esp \left[\langle\delta(X_{i}),\delta(X_{j})\rangle h^{-1}K_h( V_i - V_j )\right]\\
&=&  \esp (\langle\delta(X_{i}),\esp[\delta(X_{j})h^{-1}K_h( V_i - V_j )\mid X_{i}]\rangle)\\
&=&  \int_{[0,1]}  \esp \left(\delta(X)(t)\int_{\mathbb{R}} \exp\{2\pi i s V \} \mathcal{F}[\overline \delta(t,\cdot)  ](-s)\mathcal{F}[K](hs) ds\right)dt\\
&=&\int_{[0,1]}  \left[\int_{\mathbb{R}} \Vert \mathcal{F}[\overline \delta(t,\cdot) ](s)\Vert^2\mathcal{F}[K] (hs) ds\right] dt.
\end{eqnarray*}
When $h\rightarrow 0$, by the Lebesgue dominated convergence theorem and the Plancherel theorem applied to the integral inside the square brackets, $$\mathbb{E}[Q^\ast_{2n}(\widetilde \gamma)]\rightarrow \int_{[0,1]}\int_{[0,1]} |\overline \delta(t,v) |^2dvdt.$$
One can deduce that  $\mathbb{P}[c^{-1} \leq Q_{2n}(\widetilde \gamma) \leq c] \rightarrow 1$ for some constant $c>0.$
Taking all the results together, we can conclude that for any  $C>0,$   $\mathbb{P}[ T_n \geq C ] \rightarrow 1.$
\end{proof}

\bigskip

\begin{proof}[Proof of Corollary \ref{FPC}]
a) Let $\widehat x_{ik} = \int_{[0,1]} X_i(t)\widehat\psi_k (t) dt,$ so that $\langle X_i ,\gamma\rangle_n = \sum_{k=1}^p \widehat x_{ik} \gamma_k.$ Note that $\widehat x_{ik},$ $1\leq k \leq p,$ $1\leq i \leq n$ are measurable functions of $X_1,\cdots,X_n.$
Now, let $\widehat F_{\gamma,n}$ denote the empirical distribution function of the sample
$\langle X_{1},\gamma\rangle_n,\cdots, \langle X_{n},\gamma\rangle_n.$
Note that the elements of the matrices
$(K_h(\widehat F_{\gamma,n}(\langle X_{i},\gamma\rangle_n)- \widehat F_{\gamma,n}(\langle X_{j},\gamma\rangle_n)))$ and $(K\left( (i-j)/nh \right))$ are the same up to permutations of lines and columns. Given that in the proofs of Lemma \ref{leem1} and Theorem \ref{as_law} the arguments were provided conditionally on $X_1,\cdots,X_n,$ it  is quite clear that the conclusion of Theorem \ref{as_law} remains true if the $\langle X_{i},\gamma\rangle$'s are everywhere replaced by the $\langle X_{i},\gamma\rangle_n.$

b) Similarly, all but one of the arguments in the proof of Theorem $\ref{altern}$ applies with the $\langle X_{i},\gamma\rangle_n$'s. It only remains to investigate the counterpart of $ Q_{2n} (\widetilde \gamma)$ that was the leading term in $ Q_{n} (\widetilde \gamma).$
For this purpose, note that for any $\gamma,$
$
\langle X_i ,\gamma\rangle_n = \langle X_i ,\gamma\rangle + \langle X_i ,\Delta_{n,\gamma}\rangle
$
where
$$
\Delta_{n,\gamma} (t) = \sum_{k=1}^p\gamma_k[\widehat\psi_k (t) - \psi_k (t) ],\qquad t\in[0,1].
$$
For an integral operator $ (\Psi v)(t) =\int \psi(t,s)v(s) ds$ with $\int \int \psi^2(t,s)dt ds <\infty,$ consider the operator norm $\| \Psi \|_{S}$ defined by $\| \Psi \|^2_{S}= \int \int \psi^2(t,s)dt ds$. Under Assumption \ref{D}-(a) and the moment assumption on $\|X\|,$
$$
\|\widehat \Gamma - \Gamma \|_{S} = O_{\mathbb{P}}(1/\sqrt{n}),
$$
see for instance Bosq (2000) or  Horv\'{a}th and Kokoszka (2012).
Next, by the Cauchy-Schwarz inequality, Lemma 4.3 in Bosq (2000) or Lemma 2.3 in Horv\'{a}th and Kokoszka (2012), and the fact that the spectral norm of the operator $\widehat \Gamma - \Gamma $ is less than  or equal to  $\|\widehat \Gamma - \Gamma \|_{S}$,
$$
\int_{[0,1]} \Delta_{n,\gamma}^2 (t) dt \leq   \left[\sum_{k=1}^p\gamma_k^2 \right] \sum_{k=1}^p \| \widehat\psi_k - \psi_k \|^2 \leq p\frac{8}{\varsigma_p^2} \|\widehat \Gamma - \Gamma \|_{S}^2,
$$
where $\varsigma_p = \min_{1\leq j \leq p} (\lambda_{j} - \lambda_{j+1} ).$  Then the lower bound for the spacing between the eigenvalues implies
$$
\sup_{\gamma\in\mathcal{S}^p}\int_{[0,1]} \Delta_{n,\gamma}^2 (t) dt \leq    c p^{2\eta+1} \|\widehat \Gamma - \Gamma \|_{S}^2,
$$
for some constant $c>0$. One can deduce that
$$
\sup_{\gamma\in\mathcal{S}^p}\max_{1\leq i\leq n}\left|\langle X_i ,\gamma\rangle_n - \langle X_i ,\gamma\rangle \right|\leq
\max_{1\leq i\leq n} \| X_i\| c^{1/2}p^{\eta+1/2} \|\widehat \Gamma - \Gamma \|_{S} = O_{\mathbb{P}}(p^{\eta+1/2} \ln n /\sqrt{n}),
$$
where for the last equality we used the condition $\mathbb{E}[\exp(\varrho\|X\|)]<\infty$
to deduce that $\max_{1\leq i\leq n} \| X_i\| = O_{\mathbb{P}}(\ln n).$
Let $b_n\downarrow 0$ such that $b_n \sqrt{n}/[p^{\eta+1/2} \ln n] \rightarrow\infty$
and define the event $$\mathcal{E}_n = \{\sup_{\gamma\in\mathcal{S}^p}\max_{1\leq i\leq n}|\langle X_i ,\gamma\rangle_n - \langle X_i ,\gamma\rangle| \leq b_n\}$$ so that $
\mathbb{P}(\mathcal{E}^c_n )\rightarrow 0.
$
On the set $\mathcal{E}_n,$  for any $\gamma\in\mathcal{S}^p$ and $t\in\mathbb{R}$ we can write
$$
\widehat F_{\gamma,n}(t)= \frac{1}{n} \sum_{i=1}^n
\mathbb{I}_{\{\langle X_{i},\gamma\rangle_n\leq t  \}}
\leq \frac{1}{n} \sum_{i=1}^n
\mathbb{I}_{\{\langle X_{i},\gamma\rangle\leq t +b_n \}}= F_{\gamma,n}(t+b_n),
$$
and similarly,
$
\widehat F_{\gamma,n}(t)\geq  F_{\gamma,n}(t -  b_n).
$
One can deduce that on $\mathcal{E}_n$,
\begin{multline*}
\left| \widehat F_{\widetilde \gamma,n}(\langle X_{i},\!\widetilde \gamma\rangle_n)- \widehat F_{\widetilde \gamma,n}(\langle X_{j},\!\widetilde \gamma\rangle_n) \right|\\ \leq \max\{  \left|F_{\widetilde \gamma,n}(\langle X_{i},\!\widetilde \gamma\rangle\!+\!b_n)\!- \! F_{\widetilde \gamma,n}(\langle X_{j},\!\widetilde \gamma\rangle \!-\!b_n)\right|,\;\; \left|F_{\widetilde \gamma,n}(\langle X_{i},\!\widetilde \gamma\rangle\!-\!b_n)\!- \! F_{\widetilde \gamma,n}(\langle X_{j},\!\widetilde \gamma\rangle \!+\!b_n) \right|\}.
\end{multline*}
On the other hand,
\begin{multline*}
\left|F_{\widetilde \gamma,n}(\langle X_{i},\!\widetilde \gamma\rangle\!+\!b_n)\!- \! F_{\widetilde \gamma,n}(\langle X_{j},\!\widetilde \gamma\rangle \!-\!b_n)\right| \leq \left|F_{\widetilde \gamma,n}(\langle X_{i},\!\widetilde \gamma\rangle\!+\!b_n) - F_{\widetilde \gamma}(\langle X_{i},\!\widetilde \gamma\rangle\!+\!b_n)\right|\\
+ \left|F_{\widetilde \gamma}(\langle X_{i},\!\widetilde \gamma\rangle\!+\!b_n) - F_{\widetilde \gamma}(\langle X_{i},\!\widetilde \gamma\rangle\!-\!b_n)\right|
+\left|F_{\widetilde \gamma,n}(\langle X_{i},\!\widetilde \gamma\rangle\!-\!b_n) - F_{\widetilde \gamma}(\langle X_{i},\!\widetilde \gamma\rangle\!-\!b_n)\right|\\
\leq 2\sup_{t\in\mathbb{R}} \left|F_{\widetilde \gamma,n}(t) - F_{\widetilde \gamma}(t)\right|
+ 2 b_n \sup_{t\in\mathbb{R}}f_{\widetilde \gamma}(t)\\
=O_{\mathbb{P}}(n^{-1/2} + b_n) = O_{\mathbb{P}}(b_n).
\end{multline*}
From this and the  Lipschitz condition on $K,$ one can deduce that
$$
\left|K_h(\widehat F_{\widetilde \gamma,n}(\langle X_{i},\widetilde \gamma\rangle_n)\!- \!\widehat F_{\widetilde \gamma,n}(\langle X_{j},\widetilde \gamma\rangle_n)) - K_h( F_{\widetilde \gamma,n}(\langle X_{i},\widetilde \gamma\rangle)\!- \!F_{\widetilde \gamma,n}(\langle X_{j},\widetilde \gamma\rangle))\right| =
O_{\mathbb{P}}(b_nh^{-1}).
$$
Let $\widehat  Q_{2n} (\widetilde \gamma)$ be defined like $ Q_{2n} (\widetilde \gamma)$ but with $\langle X_{i},\widetilde \gamma\rangle$'s replaced by $\langle X_{i},\widetilde \gamma\rangle_n$'s. One can deduce from above
$$
\left|\widehat  Q_{2n} (\widetilde \gamma )-   Q_{2n} (\widetilde \gamma)\right|\leq  O_{\mathbb{P}}(b_nh^{-2}) \left[\frac{1}{n}\sum_{1\leq i\leq n}\|\delta(X_{i})\| \right]^2 = o_{\mathbb{P}}(1),
$$
provided $b_n \sqrt{n}/[p^{\eta+1/2} \ln n] \rightarrow\infty$ and $b_n=o(h^2).$ The conclusion follows.
\end{proof}

\bigskip

\bigskip

\begin{proof}[Proof of Theorem \ref{bbot}]
Consider the event $\mathcal{A}_n = \{ \max_{1\leq i \leq n} \|U_i\| \leq M \}$ with $M=n^{1/4-a} $ for some small $a.$ Assumption \ref{D}-(a)  guarantees  $\mathbb{P}(\mathcal{A}_n^c)\rightarrow 0.$
We define
\begin{equation*}\label{hn_here2}
h^b_{ i, j} = \frac{\zeta_{i} \; \zeta_{j}   }{n(n-1)h} C_{n,ij},
\end{equation*}
where $$C_{n,ij}=
\frac{\langle U_i \mathbb{I}_{\{ \| U_i \| \leq M \}}
, U_j \mathbb{I}_{\{ \| U_j \| \leq M \}}\rangle }{M^{2}} K_{h}\left( F_{\gamma,n}(\langle x_{i},\gamma\rangle)- F_{\gamma,n}(\langle x_{j},\gamma\rangle) \right).$$
Let $Q_{n}^b\left(\gamma \right)$ be the bootstrap version of $Q_{n}\left(\gamma \right),$ and let
\begin{equation*}
Q_{M,n}^b\left(\gamma \right) =\frac{1}{n(n-1)}\sum\limits_{1\leq i\neq
j\leq n}h^b_{ i, j},\quad \gamma\in\mathcal{S}^p.
\end{equation*}
Note that for any $t>0$
\begin{equation}\label{zaer}
\mathbb{P}\left[ \sup_\gamma \left|M^{-2} Q_{n}^b\left(\gamma \right) - Q_{M,n}^b\left(\gamma \right)  \right|> t \mid U_1,X_1,\cdots,U_n,X_n\right]\leq \mathbb{P}(\mathcal{A}_n^c)\rightarrow 0.
\end{equation}
We define the quantities $A_n^b$, $B_n^b,$ $C_n^b$ and $D_n^b$ as in (\ref{norms_glz1})-(\ref{norms_glz2}) with $h_{i,j}$ replaced by $h_{i,j}^b$ and the expectations replaced by the conditional expectations given $(U_1,X_1),\cdots,(U_n,X_n)$. \color{black} Since the variables $\zeta_i$ are bounded, it is easy to check that  the same upper bounds as in Lemma \ref{upp_bds} could be derived on the event $\mathcal{A}_n$. Moreover,  the condition $\mathbb{E}(U\mid X)=0,$ a.s. was not required for deriving that bounds. Hence the arguments are valid both under the null hypothesis and the alternative hypotheses.  \color{black} Then equation (\ref{zaer}) and the exponential inequality from Lemma \ref{gine_zinn} applied as in Lemma \ref{leem1} yield, for any $C>0,$
\begin{equation}\label{boot_inq}
\mathbb{P}\left[ \sup_\gamma \left| Q_{n}^b\left(\gamma \right)  \right|> Cp\ln n /nh^{1/2} \mid U_1,X_1,\cdots,U_n,X_n\right] \rightarrow 0\;\;\text{ in probability }.
\end{equation}
The second part of Lemma \ref{leem1} follows from similar arguments.

Next, let us consider the case where the null hypothesis $H_0$ holds true.
Let $\widehat \gamma_n^b$ be a least favorable direction $\gamma$ for $H_0$ obtained as in equation (\ref{bet}) using a bootstrap sample.
We deduce that $\mathbb{P}( \widehat \gamma_n^b \neq \gamma_0^{(p)}\mid U_1,X_1,\cdots,U_n,X_n) \rightarrow 0,$ in probability. It remains to
reconsider the steps of Theorem \ref{as_law} with $\widetilde Z_i$ replaced by $\zeta_i$ and $K_{h,ij}(\gamma_0^{(p)})$  by $K_{h,ij}^b(\gamma_0^{(p)}) = \langle U_i,U_j \rangle K_{h,ij}(\gamma_0^{(p)})$. Consequently $ W_{ij} $ becomes
$ W_{ij}^b = n^{-1}(n-1)^{-1}h^{-1}\zeta_i\zeta_j K_{h,ij}^b(\gamma_0^{(p)}), $
$\sigma_{ij}^2$ is replaced by $ (\sigma_{ij}^b)^2 = [K_{h,ij}^b(\gamma_0^{(p)})]^2/[n^2(n-1)^2 h^2] $ and $\sigma(n)^2 $ is now $\sigma^b( n)^2 = 2 \sum_{i\neq j} (\sigma_{ij}^b)^2$.
Let us define the set
$
\mathcal{E}_{1n} = \{ \sigma^b( n)^2 \geq \underline{\sigma}^2 \}
$
where $\underline{\sigma}^2$ is the lower bound in Assumption \ref{D}-(c)-(i). Since $\lim_n \mathbb{E}( \sigma^b( n)^2)\geq 2\underline{\sigma}^2$ and
the variance of $ \sigma^b( n)^2$ tends to zero, we deduce that $\mathbb{P}(\mathcal{E}_{1n}^c)\rightarrow 0.$
Next, note that
$\lim_n \mathbb{E}( \sigma^b( n)^2)\leq  C$
where $C$ is some constant that depends on $C_2$ and $\nu$ defined in Assumption \ref{D}-(c)-(ii). Thus, if
$
\mathcal{E}_{2n} = \{ \sigma^b( n)^2  \leq 2C \},
$
since the variance of $ \sigma^b( n)^2$ tends to zero, we have $\mathbb{P}(\mathcal{E}_{2n}^c)\rightarrow 0.$
On $\mathcal{E}_{1n} \cap \mathcal{E}_{2n} ,$ Conditions 1 and 2 of Theorem 5.1 of de Jong (1987) are clearly satisfied, given $(U_1,X_1),\cdots,(U_n,X_n)$.
For checking Condition 3, let $\mathcal{K}^b$ denote the matrix with generic element $\mathcal{K}^b_{ij} = K_{h,ij}^b(\gamma_0^{(p)})/[n(n-1)h]$ if $i\neq j$ and $\mathcal{K}^b_{ij} =0$ otherwise.
Recall that $\mathbb{E}_n$ stands for the conditional expectation given $X_1,\cdots,X_n$ and note that $\mathbb{E}_n(\|U_i\|\|U_j\|)\leq
\mathbb{E}^{1/2}(\|U_i\|^2\mid X_i) \mathbb{E}^{1/2}(\|U_j\|^2\mid X_j)\leq C_2^{2/\nu}.$  Using the conditional independence between any $U_i$ and the rest of the sample, for any $w\in\mathbb{R}^n$ with $\|w\|=1,$
\begin{eqnarray*}
\mathbb{E}_n\left\Vert \mathcal{K}^b w\right\Vert ^{2} &\leq &
\color{black}\frac{1}{h^2n^2(n-1)^2} \sum_{i=1}^{n}\mathbb{E}(\|U_i\|^2\mid X_i) \mathbb{E}_n\left(
\sum_{j=1,j\neq i}^{n} \|U_j\|  K_{h,ij}(\gamma_0^{(p)})  |w_j| \right)^2 \color{black}\\
&& \qquad\qquad \quad\qquad\qquad \quad\qquad\qquad \quad \text{[Cauchy-Schwarz inequality]}\\
&\leq & \frac{C_2^{4/\nu}}{h^2n^2(n-1)^2} \sum_{i,j,k=1}^{n} K_{h,ij}(\gamma_0^{(p)}) K_{h,ik}(\gamma_0^{(p)}) |w_{j}w_k| \notag \\
&\leq & C_2^{4/\nu}K^2 (0) \frac{1}{h^2n(n-1)^2} \sum_{j,k=1}^{n} |w_{j}w_k| \\
&\leq & \frac{C_3}{n^2} \frac{1}{nh^2},\qquad \qquad\qquad \qquad \qquad \qquad\text{[Cauchy-Schwarz inequality]}
\end{eqnarray*}
where $C_3>0$ is some constant. We deduce that $\mathbb{E}\left\Vert \mathcal{K}^b w\right\Vert ^{2}=o(n^{-2}).$ Let $| \|  \mathcal{K}^b \| |_2$ denote the spectral norm of $\mathcal{K}^b$ and  define
$
\mathcal{E}_{3n} = \{ | \|  \mathcal{K}^b \| |_2 \leq 1/n \}.
$
We deduce from the above that
$\mathbb{P}(\mathcal{E}_{3n}^c)\rightarrow 0,$ and  thus the conditional probability of $\mathcal{E}_{3n}^c$ given $ (U_1,X_1),\cdots,(U_n,X_n)$ also tends to zero.
Condition 3 in Theorem 5.1 of de Jong (1987) is clearly satisfied on $\mathcal{E}_{3n}$ and hence de Jong's CLT could be applied, \color{black}given $ (U_1,X_1),\cdots,(U_n,X_n),$ \color{black} on the event $\mathcal{E}_{n} =\mathcal{E}_{1n} \cap \mathcal{E}_{2n}\cap \mathcal{E}_{3n}$ which has a probability tending to one.
Finally, it remains to note that equation (\ref{var_esty}) holds on $\mathcal{E}_{n},$ \color{black} given $ (U_1,X_1),\cdots,(U_n,X_n).$ \color{black} The arguments for the test statistic built with the estimated FPC basis (that is not changed in the bootstrap procedure) are similar.

Let us now consider the case where the null hypothesis does not hold true. Let $\widehat v_n^{b,2} (\gamma)$ be the variance estimator obtained by bootstrapping. Since $\zeta_i^2\geq (\sqrt{5} - 1)^2/4$ and $\langle U_i^b, U_j^b \rangle^2 > \langle U_i^b, U_j^b \rangle^2 (\sqrt{5} - 1)^4/16,$ by the second part of Lemma \ref{leem1},
\begin{equation*}
\sup_{\gamma} \{1/\widehat v_n^{b,2} (\gamma)\} = O_{\mathbb{P}}(1).
\end{equation*}
From this and the bound in equation (\ref{boot_inq}), we deduce
$
T_n^b = O_{\mathbb{P}}(p\ln n )
$
given  the original sample of $(U,X),$ in probability. That means that for any $C>0,$ $\mathbb{P} ( |T_n^b|>C\mid U_1,X_1,\cdots,U_n,X_n)\rightarrow 0,$ in probability.
Consequently, $T_n^b/\alpha_n = o_{\mathbb{P}}(1)$ given the original sample of $(U,X),$ in probability. In particular, this implies that for any $a\in(0,1)$, $z^b_{1-a,n}/\alpha_n = o_{\mathbb{P}}(1).$
On the other hand, we proved in Theorem \ref{altern} that  $T_n$ tend to infinity at the rate  $O_{\mathbb{P}}(nh^{1/2} r_n^2),$ which implies that $T_n/\alpha_n$ tends to infinity, in probability. Thus, the second statement of Theorem \ref{bbot} follows. Again, the arguments for the test statistic built with the estimated FPC basis  are similar.
\color{black}
\end{proof}

\bigskip

\bigskip

\color{black}
\noindent \textbf{Lemma A.}
Under the conditions of Lemma 3.1, for any $p\geq 1$ and $\gamma \in\mathcal{S}^p,$
$$
\mathbb{E}(U\mid \langle X,\gamma\rangle ) = \mathbb{E}(U\mid F_\gamma (\langle X,\gamma\rangle) ) \;\;a.s.
$$

\medskip

\begin{proof}[Proof of Lemma A]
It suffice to prove that for any $U$ as in Lemma 3.1 and any $Z$ a real-valued random variable with distribution function $F,$ we have
\begin{equation}\label{cdddd223}
\mathbb{E}(U\mid Z) = \mathbb{E}(U\mid F(Z)) \quad a.s.
\end{equation}
For any random variable $Z$ (not necessarily continuous) with distribution function $F$ we have $\mathbb{P}(Q( F(Z) )\neq Z) = 0$ where $Q(t)=\inf \{y : F(y)\geq t \}, \forall 0<t<1.$ (See, for instance Proposition 3, Chapter 1 in Shorack and Wellner (1986).) From this and the properties of the  conditional expectations,  for any bounded measurable function $g$ we have
\begin{multline*}
\mathbb{E}(g(Z)\mathbb{E}(U\mid Z))= \mathbb{E}(g(Z)U) \\
=  \mathbb{E}(g(Q( F(Z) ))U) = \mathbb{E}(g(Q( F(Z) ))\mathbb{E}(U\mid F(Z))) = \mathbb{E}(g(Z)\mathbb{E}(U\mid F(Z))).
\end{multline*}
Since $\mathbb{E}(U\mid F(Z))$ is a measurable function of $Z,$ the almost sure uniqueness of the conditional expectation implies
the equality in equation (\ref{cdddd223}).
\end{proof}

\bigskip


Now, let us provide some theoretical justification for the sequential numerical algorithm for searching an approximation of the  direction $\widehat \gamma_n$  defined in equation (\ref{bet}). A justification is necessary only in the case of alternative hypotheses to ensure that a vector $\widetilde \gamma $ like in Theorem \ref{altern}-(e) exists. No additional justification is required for the asymptotic results on the null hypothesis since Lemma \ref{beta} still holds with $\widehat \gamma_n$ replaced by the solution obtained through our sequential numerical algorithm.

It follows from the proof of Lemma \ref{lem1}-(B) that, if
$$
\mathbb{P}\left(  \mathbb{E}\left[U\mid \langle X,\psi_1 \rangle,\cdots,\langle X,\psi_p \rangle  \right] =0\right) <1,
$$
then the set
$$
\mathcal{A}_p = \{\gamma\in\mathcal{S}^p : \mathbb{E}(U \mid \langle X, \gamma \rangle )=0 \,\, a.s.\, \}
$$
has Lebesgue measure zero on the unit hypersphere  $\mathcal{S}^p$  and is not dense. See also Lemma 2.1 in Patilea, Saumard and Sanchez (2012).

Let us next investigate what could happen when searching a direction in $\mathcal{S}^{p+1}$ using a direction in $\mathcal{S}^{p} $ and one-dimensional optimization. For a vector $v\in\mathbb{R}^p,$ let $(v,1)\in\mathbb{R}^{p+1}$ denote the vector obtained by adding an additional component equal to 1. Moreover,  let $\mathbf{0}_p$ denote the null vector in $\mathbb{R}^p.$ We prove in the following result that if $\gamma\not\in \mathcal{A}_p,$ then there exists at most a finite number of angles $\theta\in[0,\pi)$ such that $$\cos\theta \cdot (\gamma,0)+ \sin\theta \cdot (\mathbf{0}_p,1)\in \mathcal{A}_{p+1}\subset \mathcal{S}^{p+1}. $$ Moreover, if $\gamma\in \mathcal{A}_p$ and
$\mathbb{E}(U\mid \langle X,\gamma \rangle, \langle X,\psi_{p+1} \rangle)\neq 0$ (for instance, this happens when  $\mathbb{E}(U\mid \langle X,\psi_{p+1} \rangle)\neq 0$), one could draw the same conclusion on the cardinality of the set of $\theta$ such that $\cos\theta \cdot (\gamma,0)+ \sin\theta \cdot (\mathbf{0}_p,1)\in \mathcal{A}_{p+1}.$

\bigskip

\noindent \textbf{Lemma B.}
Assume that $\mathbb{E}\|U\| <\infty,$ $\mathbb{E}(U)=0$ and there exists $s>0$ such that $\mathbb{E}(\|U\|\exp(s\|X\|)) <\infty.$ If either $\gamma\in \mathcal{S} ^p\setminus \mathcal{A}_p,$ or $$\gamma\in \mathcal{A}_p \quad \text{ and } \quad \mathbb{E}(U\mid \langle X,\gamma \rangle, \langle X,\psi_{p+1} \rangle)\neq 0,$$
then the set of $\theta\in[0,\pi)$ such that $$\cos\theta \cdot (\gamma,0)+ \sin\theta \cdot (\mathbf{0}_p,1)\in \mathcal{A}_{p+1}\subset \mathcal{S}^{p+1} $$
is empty or finite.

\bigskip

\begin{proof}[Proof of Lemma B]
If $\gamma\not\in \mathcal{A}_p,$ then $\mathbb{E}(U\mid \langle X,\gamma \rangle)\neq 0.$ In particular, we have  $\mathbb{E}(U\mid \langle X,\gamma \rangle, \langle X,\psi_{p+1} \rangle)\neq 0$ and thus the arguments below apply for both cases in the statement of Lemma B.
Now, if $\theta^*\in (0,\pi)$ is such that
\begin{equation}\label{thetoo}
\mathbb{E}(U\mid \cos \theta^* \cdot  \langle X,\gamma \rangle + \sin\theta^* \cdot \langle X,\psi_{p+1} \rangle)= 0,
\end{equation}
then, for any $b\in \mathbb{R},$ $\theta^*$ is a zero of the analytic function
\begin{equation}\label{analytic1}
\theta\mapsto \zeta_b(\theta) = \mathbb{E}\left[U \exp( b\{ \cos \theta\cdot  \langle X,\gamma \rangle + \sin\theta \cdot \langle X,\psi_{p+1} \rangle \}) \right],\qquad \theta\in(0,2\pi).
\end{equation}
(Since $X$ could be rescaled conveniently if necessary, we consider $\mathbb{E}(\|U\|\exp(s\|X\|)) <\infty$ for some $s\geq 1$. This guarantees that the expectation in the  display (\ref{analytic1}) is well defined.) Since $\mathbb{E}(U\mid \langle X,\gamma \rangle, \langle X,\psi_{p+1} \rangle)\neq 0,$ there exists $b$ such that the analytic function $\zeta_b(\cdot)$ is not the null function. Otherwise, by the unicity of the Fourier Transform, necessarily  $\mathbb{E}(U\mid \langle X,\gamma \rangle, \langle X,\psi_{p+1} \rangle)= 0$ a.s. Next, the zeros of a non-null analytic function in a bounded interval are necessarily in finite number. Hence, the set of $\theta^*$ satisfying equation (\ref{thetoo}) is necessarily finite.
\end{proof}
\color{black}


\newpage

\begin{center}
{\small ADDITIONAL REFERENCES }
\end{center}


\begin{enumerate}
\item {\footnotesize \textsc{Bosq, D.} (2000). \textsl{Linear Processes in Function Spaces: Theory and Applications}. Lecture Notes in Statistics (v. 149). Springer-Verlag, New-York.}
    \item {\footnotesize \textsc{Guerre, E., and Lavergne, P.} (2005).
Data-driven rate-optimal specification testing in regression models. \textsl{%
Ann.  Statist.} \textbf{33}, 840--870. }

\item {\footnotesize \textsc{Shorack, G.R., and Wellner, J.A.} (1986). \emph{Empirical Processes with Applications to Statistics.} John Wiley \& Sons.}

\end{enumerate}

\end{document}